\documentclass[12pt]{amsart}

\usepackage[margin=1in]{geometry}

\usepackage{float}

\usepackage{setspace}
\usepackage[toc,page]{appendix}
\usepackage{datetime}
\usepackage[T1]{fontenc}
\usepackage{chngcntr}
\usepackage[T1]{fontenc}
\usepackage{amssymb,url,xspace,amsthm,adjustbox}
\usepackage[alphabetic,nobysame]{amsrefs}
\usepackage{mathtools}
\usepackage{graphicx}
\usepackage{amssymb}
\usepackage{mathrsfs,stmaryrd}
\usepackage{yfonts, bbm}
\usepackage{enumitem}

\newcommand\norm[1]{\lVert#1\rVert}

\usepackage{hyperref}
\hypersetup{
	colorlinks   = true, 
	urlcolor     = blue, 
	linkcolor    = blue, 
	citecolor   = green 
}

\usepackage{tikz}
\usetikzlibrary{shapes,arrows,calc,matrix}
\usepackage{tikz-cd}
\usepackage{todonotes}
\usepackage{color}

\usepackage{amsmath}
\usepackage{lipsum}
\usepackage{setspace}

\theoremstyle{plain}
\newtheorem{theorem}{Theorem}[section]
\newtheorem{proposition}[theorem]{Proposition}
\newtheorem{lemma}[theorem]{Lemma}
\newtheorem{corollary}[theorem]{Corollary}
\newtheorem{conjecture}[theorem]{Conjecture}

\theoremstyle{definition}
\newtheorem{definition}{Definition}

\newtheorem{example}[theorem]{Example}

\newtheorem*{conjecture*}{Conjecture}

\usepackage{scalerel, stackengine}
\stackMath
\newcommand\rwh[1]{%
	\savestack{\tmpbox}{\stretchto{%
			\scaleto{%
				\scalerel*[\widthof{\ensuremath{#1}}]{\kern-.6pt\bigwedge\kern-.6pt}%
				{\rule[-\textheight/2]{1ex}{\textheight}}
			}{\textheight}%
		}{0.5ex}}%
	\stackon[1pt]{#1}{\tmpbox}%
}
\newcommand{\mc}{\mathcal}
\newcommand{\mb}{\mathbb}
\newcommand{\mf}{\mathfrak}

\newcommand{\la}{\left\langle}
\newcommand{\ra}{\right\rangle}
\newcommand{\lset}{\left\lbrace}
\newcommand{\rset}{\right\rbrace}
\newcommand{\w}{\wedge}

\newcommand{\lb}{\left(}
\newcommand{\rb}{\right)}

\newcommand{\bc}{\mathbb{C}}

\newcommand{\bn}{\mathbb{N}}

\newcommand{\bq}{\mathbb{Q}}
\newcommand{\br}{\mathbb{R}}
\newcommand{\bz}{\mathbb{Z}}

\newcommand{\tx}{\text}
\newcommand{\ten}{\otimes}

\newcommand{\Hom}{\mathrm{Hom}}
\newcommand{\Ext}{\mathrm{Ext}}

\newcommand{\End}{\mathrm{End}}

\newcommand{\Ad}{\mathrm{Ad}}

\newcommand{\For}{\mathrm{For}}
\newcommand{\Per}{\mathrm{Per}}
\newcommand{\Loc}{\mathrm{Loc}}

\newcommand{\Rep}{\mathrm{Rep}}
\newcommand{\Mod}{\mathrm{Mod}}

\newcommand{\1}{\mathbbm{1}}
\newcommand{\GL}{\mathrm{GL}}
\newcommand{\SL}{\mathrm{SL}}
\newcommand{\PGL}{\mathrm{PGL}}
\newcommand{\rec}{\mathrm{rec}}
\newcommand{\ve}{\varepsilon}
\newcommand{\BM}{\mathrm{BM}}

\author[K. J. Balodis]{Kristaps John Balodis}
\address{Department of Mathematics and Statistics, University of Calgary, 
	2500 University Drive NW, 
	Calgary, Alberta, 
	T2N 1N4, 
	Canada}
\email{kristaps.balodis1@ucalgary.ca}

\counterwithin{table}{subsection}

\makeindex

\title[The $p$-kLH for $\GL(n)$]{Proof of the $p$-adic Kazhdan-Lusztig hypothesis for $\GL(n)$}

\date{\today}

\begin{document}
	
	\maketitle

	\begin{abstract}
		In this article, we prove the $p$-adic Kazhdan-Lusztig hypothesis for $\mathrm{GL}_n(F)$, as first put forward by \cite{Zel}, and later generalized by \cite{Vog}.  
		The proof given here relies  on  the celebrated theorem \cite{CG}*{Theorem 8.6.23} which establishes an analogue of the Kazhdan-Lusztig hypothesis for affine Hecke algebras. 
		The main work in this paper is in carefully tracking through various equivalences of categories and isomorphisms to ensure simple and standard objects are matched appropriately. 
		Along the way we establish several folkloric results regarding standard representations/modules, and the combinatorics of Zelevinsky multisegments. 
	\end{abstract}
	
	\tableofcontents
	
	\section{Introduction}
	
	\subsection{Overview and main results}
	
	The main goal of this paper is to establish the $p$-adic Kazhdan-Lusztig hypothesis for $\mathrm{GL}_n(F)$, where $
	F/\mb{Q}_p$ is a
	$p$-adic field, as originally articulated by \cite{Zel}, and later generalized by \cite{Vog}. Roughly speaking, the $p$-adic Kazhdan-Lusztig hypothesis predicts that multiplicities of smooth irreducible representations in standard representations can be computed via stalks of equivariant perverse sheaves on moduli spaces of Langlands parameters.
	
	To be more precise, we will freely use standard terminology in the local Langlands program.
	For simplicity, we let $G$ be the $F$-points of a split, connected, reductive group over a $p$-adic field $F$, and $\hat{G}$ its complex dual group.  
	The reciprocity map of the local Langlands correspondence associates to every smooth irreducible representation $\pi$ of a $p$-adic group $G(F)$, a group homomorphism $\phi:W_F\ltimes \bc \to \hat{G}(\bc)$, for the Weil group $W_F$.  
	The infinitesimal parameter is defined to be $\lambda :=\phi|_{ W_F }$ .  
	One can form a variety $V_\lambda$ (Equation \ref{eqn: Vl}), together with an action of an algebraic group $H_\lambda$ (Equation \ref{eqn: Hl}), where the (finitely many) orbits are in bijection with the Langlands parameters having infinitesimal parameter $\lambda$.  
	The local Langlands correspondence predicts that there is a bijection between smooth irreducible representations of infinitesimal parameter $\lambda$, and pairs $(C, \mc{L})$ where $C$ is an $H_\lambda$-orbit of $V_\lambda$ and $\mc{L}$ is an irreducible $H_\lambda$-equivariant local system on $C$.  
	Writing $\pi(C, \mc{L})$ for the unique irreducible representation associated to $(C, \mc{L})$ (assuming Langlands reciprocity for $G$) and $S(C, \mc{L})$ for the standard representation of $\pi(C, \mc{L})$ (Definition \ref{def: lc}),
	the $p$-adic Kazhdan-Lusztig hypothesis predicts that:
	\begin{align}\label{eqn: pklhintro}
		[S(C, \mc{L}): \pi(D, \mc{F})]&=\sum_{n\in \bz}[\mathcal{H}^n(\mc{IC}(D, \mc{F}))|_{C} : \mc{L}]
	\end{align}
	where $[A:B]$ means the number of times $B$ occurs in the composition series of $A$, and $\mc{IC}(D, \mc{F})$ is the intersection cohomology complex associated to $\mc{F}$. 
	
	\subsection{Historical context and motivation}
	
	The original Kazhdan-Lusztig hypothesis, formulated in \cite{KL} for a semisimple complex Lie group $G$ with Lie algebra $\mf{g}$, predicted that multiplicities of simple $\mf{g}$-modules in Verma modules can be computed in terms of dimensions of stalks of perverse sheaves on $G/B$ where $B$ is a Borel subgroup, and was proved by Beilinson-Bernstein \cite{BB} and Brylinski-Kashiwara \cite{BryKas}. 
	Moreover, it was shown that the stalks of the perverse sheaves on $G/B$ can be described by generating functions known as Kazhdan-Lusztig polynomials, which admit a purely combinatorial description.
	Thus, while computing multiplicities directly at the module level is difficult, it can be reduced to a combinatorial problem solvable by a computer algebra system. 
	
	The Kazhdan-Lusztig hypothesis for real groups was proved in \cite{VogProof}.   
	This led to the Atlas project, enabling the first complete understanding of representations of the split form of 
	$E_8$, which would have been computationally infeasible without Kazhdan-Lusztig theory.
	
	Given all of the above, one should expect that the $p$-adic version, which we establish for $\GL_n(F)$, may offer similar computational (and theoretical) advantages.
	For example, \cite{BCFZ} establishes that the forwards direction of some conjectures of Gross-Prasad and Shahidi follow for groups which satisfy the $p$-adic Kazhdan-Lusztig hypothesis.

	\subsection{The approach in this paper}

	The proof strategy in this paper builds on the celebrated theorem of \cite{CG}*{Theorem 8.6.23}, (see Equation \ref{eqn: cgpklh}) which states that 
	$$[H^\bullet(\mc{B}_x^s)_\xi:L_\gamma]=\sum_{k\in \bz}\dim \mc{H}^k(i_x^\ast P_\gamma)_\xi=\sum_{k\in \bz}[ \mc{H}^k(i_x^\ast P_\gamma):\mc{L}_\xi].$$
	
	The $P_\gamma$ and $L_\xi$ are respectively intersection cohomology complexes, and local systems on what is essentially a Vogan variety $V_\lambda$, though it is not described there as such.  
	Meanwhile, the left hand side is a statement about multiplicities of modules over an algebra $E_{s, q}$ (see Section 3.1), constructed from the geometry of $V_\lambda$.  (The indices $s, q$ are determined by $\lambda$.)   
	The above equation is of course extremely similar to Equation \ref{eqn: pklhintro}, however connecting the equations requires overcoming several technical obstacles.
	
	The variety $V_\lambda$ is constructed from a complex Lie group $\hat{G}(\bc)$, 
	and strictly speaking \cite{CG}*{Theorem 8.6.23} only applies when $\hat{G}(\bc)$ is semisimple and simply connected.  
	However, in order to derive the $p$-adic Kazhdan-Luztig hypothesis for $\GL_n(F)$, we need to take $\hat{G}(\bc)=\GL_n(\bc)$ which is neither semismple nor simply connected. 
	Fortunately, each of the varieties $V_\lambda$ for $\GL_n(\bc)$ are isomorphic to a variety for $\SL_n(\bc)$.  
	The groups $H_\lambda$ and the algebra $H_\lambda$ differ, but in a way that is straightforward to manage.  
	
	In order to apply \cite{CG}*{Theorem 8.6.23} to representations of $G=\GL_n(F)$, one must also understand how the algebras $E_{s, q}$ are related to the Iwahori-Hecke algebra $\mc{H}(G, \mc{I})$ (isomorphic to an affine Hecke algebra).
	Every irreducible module of $\mc{H}(G, \mc{I})$ admits a central character.  
	Given such a central character $\chi$, let $K_\chi$ be the (left) ideal generated by $\ker(\chi)$, then $\pi$ determines an $\mc{H}(G, \mc{I})/K_\chi$-module.  
	For every pair $(s, q)$, there exists a central character $\chi$ such that $ E_{s, q}\cong \mc{H}(G, \mc{I})/K_\chi$, and thus by extending scalars along the projection 
	$$\varphi:\mc{H}(G, \mc{I})\to \mc{H}(G, \mc{I})/K_\chi \cong  E_{s, q},$$
	every irreducible $E _\lambda$-module determines an irreducible $\mc{H}(G, \mc{I})$-module, and by \cite{CG}{Theorem 8.1.16} every  irreducible  $\mc{H }(G , \mc{ I} )$-module  arises  in this way.   
	However, it is unclear given one of the simple $E_{(s, q)}$-modules $L_\xi$, which $\mc{H}(G, \mc{I})$-module $\varphi_\ast(L_\xi)$ is, and therefore which irreducible representation of $\GL_n(F)$ it corresponds to.  
	Likewise, it is unclear that the modules $H^\bullet(\mc{B}_x^s)_\xi$ should be sent to standard representations.  
	We will not actually carry out this precise comparison here.    
	Instead, we rely on the work of \cite{Ari} which only relates these modules on the level of the Grothendieck group, as this is sufficient to prove the $p$-adic Kazhdan-Lusztig hypothesis.  It would be interesting to find a proof that establishes the result on the level of modules. 
	
	Even if one were able to account for the bookkeeping of modules as above, the results of \cite{CG} only account for modules over affine Hecke algebras, and therefore in general can only resolve the $p$-adic Kazhdan-Lusztig hypothesis for representations belonging to the principal block of $\mathrm{Rep}(\mathrm{GL}_n(F))
	$.
	However, it is known that every Bernstein block $\Rep(\GL_n(F))_\tau$ of $\Rep(\GL_n(F))$ is equivalent to a category of modules over an affine Hecke algebra (and thus to the principal block) of some $\GL_m(E)$ for an integer $m<n$, and a $p$-adic field $E$. 
	That being said, in order leverage this fact, and the $p$-adic Kazhdan-Lusztig hypothesis for the principal block in order to obtain the general result, one must carefully track where irreducible and standard representations are sent to under the equivalence with a principal block, as well as prove corresponding relations between the associated Vogan varieties, which is done in this article.

	\subsection{Related work}
	
	The result \cite{Solp}*{Theorem 5.4} is an analogue of the $p$-adic Kazhdan-Lusztig hypothesis, but for modules over \emph{graded} affine Hecke algebras.  
	In conjunction with \cite{Sol2}*{Lemma 6.2}, which ensures that standard \emph{modules} of graded affine Hecke algebras correspond to standard \emph{representations}, the $p$-adic Kazhdan-Lusztig hypothesis as stated by \cite{Vog}, has been proven in the cases listed in \cite{Solp}*{Theorem 5.4}.  
	In particular, the $p$-adic Kazhdna-Lusztig hypothesis has not only been proven for $\GL_n(F)$, but for a wider class of groups and representations that one could obtain from applying \cite{CG}*{Proposition 8.6.23}.  
	The present article still exists  as it makes clear the connection between the version for affine Hecke algebras and representations of $p$-adic groups.  
	We also mention that the $p$-adic Kazhdan-Lusztig hypothesis has been verified for unipotent representations of $G_2$ in \cite{CFZ:cubic} and \cite{CFZ:unipotent}.

	The author would like to thank Maarten Solleveld and Chris Jantzen for their help, and a special thanks to Clifton Cunningham, who supervised this project. 
	

	\section{Representation Theory of $\GL_n(F)$}
	
	In this section we prepare a number of technical results about the representation theory of $\GL_n(F)$.  
	In particular, in Section 2.2 we prove Theorem \ref{thm: equiv_seg_stand} that the representations $S_Q(a)$ defined in terms of multisegments are in fact the representations described by the Langlands classification theorem.  
	Pairing Theorem \ref{thm: equiv_seg_stand} with Proposition \ref{prop: zmult=qmult} (proved in Section 2.3) will allow us to compare \cite{Zel}*{Hypothesis 1.9} and \cite{Vog}*{Conjecture 8.11}.  
	In Section 2.4 we establish Proposition \ref{prop: prodmult} which allows one to decompose general multisegments in terms of simpler ones, which will be used in deriving the general case of the $p$-adic Kazhdan-Lusztig hypothesis from the case of representations $S_Q(a), Q(b)$ where $a, b$ have \emph{simple} inertial support (Definition \ref{def: simp}). 
	One of the main results of this section is Theorem \ref{thm: heckemod} which characterizes the modules over a version of the affine Hecke algebra described by \cite{Ari}, which correspond to standard representations, or rather, to the representations $S_Z(a)$. 
	
	\subsection{Standard representations}
	
	In this section we summarize the Langlands classification theorem as it appears in \cite{Kon}.  
	The result of the classification theorem is crucial for defining standard representations, which are a key component of the $p$-adic Kazhdan-Lusztig hypothesis.  
	
	Let $G$ be the $F$-points of a connected reductive algebraic group defined over $F$, fix a maximal $F$-split torus $T_0$ of $G$, and let $M_0=C_G(T_0)$ which is a minimal Levi subgroup of $G$.  
	For any Levi subgroup $M$, we let $X^\ast(M)_F$ denote the group of $F$-rational characters, and set
	\begin{align*}
		\mf{a}_M&:=\Hom(X^\ast(M)_F, \br) \\
		\mf{a}_M^\ast&:=X^\ast(M)_F\ten\br. 
	\end{align*}
	
	Suppose $M$ is a minimal Levi subgroup, and let $\Sigma$ be the roots of the maximal torus $A$ of $M_0$. 
	Each parabolic $P$ containing $M$ determines a set $\Sigma_P$ of $P$-positive elements.  
	A parabolic $P$ is said to be \emph{standard} if it contains $M_0$.  
	A minimal standard parabolic $P_0$  determines a set of simple roots $\Delta_{P_0}$ in $\Sigma_{P_0}$.  
	
	For any standard parabolic $P$ with Levi decomposition $P=MU$, define
	\begin{align*}
		\Delta_P&:=\{\alpha_0|_{\mf{a}_M} : \alpha_0\in \Delta_{P_0}-\Delta_{P_0\cap M}\},\\
		\mf{a}_P^{\ast, +}&:=\{\lambda \in\mf{a}_M^\ast : \forall \alpha\in \Delta_P, \alpha^\vee(\lambda)>0\}.
	\end{align*}
	
	\begin{theorem}[\cite{Kon}, Theorem 3.5]\label{thm: lc}
		For any irreducible admissible representation $(V, \pi)$ of  $G(F)$, there exists a standard parabolic $P=MU$ and an irreducible tempered representation $\tau$ of $M(F)$, and $\lambda\in\mf{a}_P^{\ast, +}$ such that $\pi$ is the unique irreducible quotient of $I_P^G(e^\lambda\ten\tau)$, where $I_P^G$ is the normalized parabolic induction functor.  
		Moreover, the triple $(P, \tau, \lambda)$ is uniquely determined up to $W$-conjugacy. 
	\end{theorem}
	
	\begin{definition}\label{def: lc}
		Fix a choice of minimal parabolic $P_0$ of $G$. 
		Given an irreducible representation $\pi$ of $G$, there is a unique representation $I_P^G(e^\lambda\ten\tau )$ determined by the above theorem such that $P$ is standard.  
		We will refer to $I_P^G(e^\lambda\ten\tau)$ as the \emph{standard representation} of $\pi$ and denote it by $S_\pi$.  
		Given an irreducible representation $\pi$, we will refer to any one of the $W$-conjugate triples $(P, \tau, \lambda)$ determined by $\pi$ by the above theorem as a \emph{Langlands triple} or \emph{Langlands data}.  
	\end{definition}


	\subsection{Between the Langlands and Zelevinsky classification theorems}

	The main result of this section is Theorem \ref{thm: equiv_seg_stand}, which establishes that $S_Q(a)$ are exactly the standard representations of Definition \ref{def: lc}.  
	In order to carry out the proof, we will need to first establish several technical results about multisegments.  
	
	We first recall the classification of smooth irreducible representations of $\GL_n(F)$ in terms of \emph{multisegments} as developed by \cite{ZelI2}. 
	Let $\nu$ be the character of $\GL_n(F)$ determined by $\nu(g)=|\det(g)|_F$.  
	For any $c\in \bc$, we write $\rho(c):=\nu^c\ten\rho=\nu^c\rho$.  
	
	\begin{definition}
		A \emph{segment} $\Delta=[\rho(a), \rho(b)]$ is an ordered set 
		$$\{\rho(a), \rho(a+1), \ldots, \rho(b)\},$$
		where $\rho$ is a supercuspidal representation.
	\end{definition} 
	
	Recall that each partition $\vec{n}=(n_1, \ldots, n_r)$ of $n=n_1+\cdots + n_r$ determines a Levi subgroup, 
	$$M_{\vec{n}}=\lset\begin{pmatrix}
		A_1 & 0 & ... & 0 \\
		0 & A _2 & ... & 0 \\
		... & & ... & ... \\
		0 & 0 & ... & A_r
	\end{pmatrix} : A_i\in \GL_{n_i}\rset,$$
	and every Levi subgroup of $\GL_n(F)$ is conjugate to a Levi subgroup of this form.  
	Let $P_{\vec{n}}$ denote the \emph{standard} parabolic subgroup, consisting of all matrices of the form
	$$\begin{pmatrix}
		A_1 & 0 & ... & 0 \\
		0 & A _2 & ... & 0 \\
		... & & ... & ... \\
		0 & 0 & ... & A_r
	\end{pmatrix}
	\begin{pmatrix}
		I_1 & N_{12} & \ldots & N_{1r} \\
		0 & I_2 & \ldots & N_{2r} \\
		\ldots & & \ldots & \ldots \\
		0 & 0 & \ldots & I_r
	\end{pmatrix}.$$
	
	Given a Levi subgroup $M_{n_1, \ldots, n_r}$ of $\GL(n)$, where $n = n_1+\cdots + n_r$, and representations $\sigma_i\in \Rep(\GL(n_i))$, we write
	$$\sigma_1\times\cdots \times \sigma_r:=I_{P_{\vec{n}}}^G\lb \sigma_1\boxtimes\cdots \boxtimes\sigma_r\rb,$$
	for the normalized parabolic induction. 
	
	To any segment $\Delta=[\rho(a), \rho(b)]$, we associate a representation
	$$S_Z(\Delta):=\rho(a)\times\rho(a+1)\times\cdots \times \rho(b),$$
	which by \cite{ZelI2}*{Proposition 2.10},  has a unique irreducible subrepresentation $Z(\Delta)$, and a unique irreducible quotient $Q(\Delta)$.   
	
	\begin{definition}
		A \emph{multisegment} is an \emph{ordered} multiset of segments.  
	\end{definition}
	
	We will make use of the following definition. 
	
	\begin{definition}
		Given a multisegment $a=\{\Delta_1, \ldots, \Delta_r\}$,  we define the respective \emph{Zelevinsky sub-standard representation} and \emph{Zelevinksy quotient-standard} representations
		\begin{align*}
			S_Z(a)&:=Z(\Delta_1)\times\cdots \times Z(\Delta_r)\\
			S_Q(a)&:=Q(\Delta_1)\times\cdots \times Q(\Delta_r). 
		\end{align*}
	\end{definition}

	\begin{definition}[Does Not Proceed]\label{dnp}
		Two segments $\Delta=[\rho_1, \rho_1'], \Delta'=[\rho_2, \rho_2']\in a$ are said to be \emph{linked} if neither is contained in the other, and their union is a segment.  
		If they are linked and $\rho_2\cong \nu^k\rho_1$ for an integer $k>0$, then we say that $\Delta_1$ \emph{precedes} $\Delta_2$. 
		If $\Delta_1$ and $\Delta_2$ are linked and $\Delta_1\cap \Delta_2=\emptyset$, then we say they are \emph{juxtaposed}.
		
		A multisegment $a=\{\Delta_, \ldots, \Delta_r\}$ is said to satisfy the \emph{Does Not Precede} condition if for $i<j$, then $\Delta_i$ does not precede $\Delta_j$. 
	\end{definition}
	
	By \cite{ZelI2}*{Theorem 6.1}, if $a$ satisfies Definition \ref{dnp}, then $S_Z(a)$ has a unique irreducible subrepresentation which we denote by $Z(a)$, and every irreducible representation appears as $Z(a)$ for some multisegment $a$ satisfying Definition \ref{dnp}.   
	Likewise, by \cite{Rod}*{Theorem 3}, if $a$ satisfies Definition \ref{dnp}, then $S_Q(a)$ contains a unique irreducible quotient denoted $Q(a)$, and every irreducible representation of $\GL_n(F)$ arises as some $Q(a)$.

	The following example is included to highlight that, while not particularly complicated, the translation between multisegments and Langlands data is not as straightforward as one might initially expect. 
	
	\begin{example}
		Let $\nu:\GL_1(\bq_p)\to \bc$ by the norm character $\nu(g)=|g|_F$ and consider the multisegment $a=\{[\nu^0, \nu^1], [\nu^0, \nu^1]\}$ of $\GL_4(F)$.  
		Since the corresponding representation $S_Q(a)$ is induced from $P_{2, 2}$, one might expect that $P_{2, 2}$ is the parabolic in the corresponding Langlands triple.  
		However, writing $\tau:=Q[\nu^{-1/2}, \nu^{1/2}]$, by Lemma \ref{lem: twist_seg} below $Q[\nu^0, \nu^1]\cong (\nu^{1/2}\circ \det )\ten\tau$, and thus
		$$S_Q(a)=I_{P_{2, 2}}^G\lb Q[\nu^0, \nu^1]\boxtimes Q[\nu^0, \nu^1]\rb\cong I_{P_{22}}^G\lb(\nu^{1/2}\boxtimes \nu^{1/2})\ten (\tau\boxtimes \tau)\rb.$$
		While $P_{2, 2}$ is a standard parabolic subgroup and $\tau\boxtimes \tau$ is tempered (see Theorem \ref{thm: temp}), $\nu^{1/2}\boxtimes \nu^{1/2}$ is not $P_{2, 2}$-poistive, and therefore e
		$(P_{2, 2}, \tau\boxtimes \tau, \nu^{1/2}\boxtimes \nu^{1/2})$ does not define a Langlands triple.  
		However, 
		\begin{align*}
			S_Q(a)&\cong I_{P_{22}}^G\lb(\nu^{1/2}\boxtimes \nu^{1/2})\ten (\tau\boxtimes \tau)\rb\\
			&\cong (\nu^{1/2}\circ \det)\ten I_{P_{22}}^G\lb \tau\boxtimes \tau\rb\\
			&\cong I_G^G\lb \nu^{1/2}\ten I_{P_{22}}^G\lb \tau\boxtimes \tau\rb\rb.
		\end{align*}
		
		The representation $I_{P_{22}}^G\lb \tau\boxtimes \tau\rb$ is irreducible and tempered by Theorem \ref{thm: temp}, and $\nu^{1/2}=e^{\mu}$ where for some $\mu\in \mf{a}_G^{\ast, +}$.
		Therefore the Langlands triple associted to $S_Q(a)$ is $(G, I_{P_{22}}^G\lb \tau\boxtimes \tau\rb, \mu)$.  
		Hence, for example, the parabolic subgroup in the triple defining the standard representation isomorphic to $S_Q(a)$ is not the parabolic defining the inducing data of $S_Q(a)$. 
	\end{example}
	
	We now establish several technical lemmas which we require for the proof of Theorem \ref{thm: equiv_seg_stand}. 
	
	\begin{definition}
		We call $\{\rho(c) | c\in \bc\}$ the \emph{inertial support} of $\Delta$, and note that (up to isomorphism) for any member $\rho'$ of the inertial support of $\Delta$, we can choose some $a', b'\in \bc$ such that $\Delta=[\rho'(a'), \rho'(b')]$. 
		If the trivial representation of $\GL_1(F)$ (which is supercuspidal) belongs to the cuspidal support of $\Delta$, then we say that $\Delta$ has \emph{trivial} inertial support.  
		We say that a multisegment has trivial inertial support, if all of its segments have trivial inertial support. 
	\end{definition}
	
	\textbf{Note:} The above is not to be confused with the cuspidal support of a \emph{representation}.  
	The concepts are distinct, but related, though there should be no chance of confusion as we will not discuss the inertial support of representations in this article. 
	
	We will write $\nu:\GL_n(F)\to \bc^\times$ be the norm-character $\nu(g)=\norm{\det g}_F$, suppressing the dependence of $n$ and $F$. 
	When $n=1$, $\nu$ is supercuspidal, and we will simply write $\Delta=[a, b]$ for the segment
	$$\Delta=\{\nu^a, \nu^{a+1}, \ldots, \nu^b\}.$$
	Note that these are exactly the segments of trivial cuspidal support.  
	
	If $\Delta, \Delta'$ have distinct inertial support, then they cannot be linked.  
	If they do have the same inertial support, then fixing a representative $\rho$ , we can write $\Delta=[\rho(a), \rho(b)], \Delta'=[\rho(c), \rho(d)]$. 
	Then, $\Delta, \Delta'$ are linked if and only if,
	\begin{enumerate}
		\item $(a-c)\in \bz$, and for their real parts
		\item $\Re(a)<\Re(c)$ and $\Re(c-1)\leq \Re(b)<\Re(d)$, or \\
		\item $\Re(c)\leq \Re(a)\leq \Re(d+1)$ and $\Re(d)<\Re(b)$.
	\end{enumerate}
	
	We define a notion of the \emph{midpoint} of a segment $\Delta=[\rho(a), \rho(b)]$, relative to the supercuspdial $\rho$ to be
	$$m_{\rho}(\Delta):=\frac{a+b}{2}.$$
	Observe that (up to isomorphism of the elements) we can write $\Delta=[\rho(c)(a-c), \rho(c)(b-c)]$, and thus
	$$m_{\rho(c)}(\Delta):=\frac{(a-c)+(b-c)}{2}=m_\rho(\Delta)-c.$$
	
	We also define the \emph{length}
	$$\ell(\Delta):=b-a+1,$$
	which, being the number of terms in $\Delta$ is independent of the actual representations it contains.

	\begin{lemma}\label{lem: can ord}
		Let $\Delta$ and $\Delta'$ be segments with inertial support determined by $\rho$. 
		If $m_\rho(\Delta)=m_\rho(\Delta')$, then $\Delta$ and $\Delta'$ are not linked, and if
		$m_\rho(\Delta)< m_\rho(\Delta')$, then $\Delta$ precedes $\Delta'$.  
		Together, this means that if $m_\rho(\Delta)\geq m_\rho(\Delta')$, then $\Delta$ does not precede $\Delta'$. 
	\end{lemma}

	\begin{lemma}\label{lem: twist_seg}
		Let $\Delta=[\rho(a_1), \rho(a_r)]$ be a segment, and $\chi$ a smooth character of $\GL_1(F)$.  
		Abusing notation, we will also write $\chi$ for the character $\chi(\det(g))$ of $\GL_n(F)$.   
		Then,
		\begin{align*}
			\chi\ten Z(\Delta)&\cong (Z[\chi\rho(a_1), \chi\rho(a_r)])\\
			\chi\ten Q(\Delta)&\cong Q([\chi\rho(a_1), \chi\rho(a_r)]).
		\end{align*}
		
	\end{lemma}
	\begin{proof}
		We will just prove the second isomorphism as the proof of the first is entirely similar. 
		
		For any segment $\Delta=[\rho(a_1), \rho(a_r)]$ and $b \in \bc$,
		$$S_Q(\Delta)=I_P^G(\boxtimes_{i=1}^r\nu^{a_i}\rho)\cong I_P^G\lb\lb\boxtimes_{i=1}^r\nu^b\rb\ten\lb \boxtimes_{i=1}^r \nu^{a_i-b}\rho\rb\rb\cong \nu^b\ten I_P^G\lb \boxtimes_{i=1}^r\nu^{a_i-b}\rho\rb.$$
		Since tensoring by characters is an exact functor, by applying it to the sequence
		$$I_B^G\lb \nu^{a_1}\rho\boxtimes\cdots \boxtimes \nu^{a_r}\rho\rb\to Q([\nu^{a_1}\rho, \nu^{a_r}\rho])\to 0,$$
		we get 
		$$\chi\ten I_B^G\lb \nu^{a_1}\rho\boxtimes\cdots \boxtimes \nu^{a_r}\rho\rb\to \chi\ten Q([\nu^{a_1}\rho, \nu^{a_r}\rho])\to 0.$$
		Given $f\in I_B^G\lb \chi\nu^{a_1}\rho\boxtimes\cdots\boxtimes  \chi\nu^{a_r}\rho\rb$, define $f_\chi(x):=\chi(x)f(x)$.  
		This determines an isomorphism,
		\begin{align*}
			I_B^G\lb \chi\nu^{a_1}\rho\boxtimes\cdots\boxtimes  \chi\nu^{a_r}\rho\rb &\to  \chi^b\ten I_B^G\lb \nu^{a_1}\rho\boxtimes\cdots \boxtimes\nu^{a_r}\rho\rb\\
			f&\mapsto f_\chi
		\end{align*}
		Therefore $[\nu^{a_1}\chi\rho, \nu^{a_r}\chi\rho]$ is still a segment, and the induced representation has unique irreducible quotient $Q([\nu^{a_1}\chi\rho, \nu^{a_r}\chi\rho])\cong \chi\ten Q([\nu^{a_1}\rho, \nu^{a_r}\rho])$.  
	\end{proof}
	
	We will make use of the following result below.  
	
	\begin{lemma}[\cite{Rod}*{Proposition 11}]\label{lem: Rod}
		A representation $Q(\Delta)$ is square-integrable if $\Delta=[\rho(a), \rho(b)]$ and  $\rho((a+b)/2)$ is unitary. 
	\end{lemma}

	\begin{lemma}\label{lem: sqr}
		Let $\rho$ be a supercuspidal representation of $\GL(n)$.  
		Then,
		\begin{enumerate}
			\item There is a unique real number $x_\rho\in \br$ such that $\rho_u:=\rho(-x_\rho)$ is unitary. 
			\item Given a segment $\Delta$, by the previous result we may choose a unitary representative $\rho$ for its inertial support. 
			Writing $\Delta=[\rho(a), \rho(b)]$, and $\bar{\Delta}=[\rho((a-b)/2), \rho((b-a)/2)]$, we have
			$$Q(\Delta)\cong \nu^{m_\rho(\Delta)}\ten Q(\bar{\Delta}),$$
			where $Q(\bar{\Delta})$ is square-integrable. 
			\item $m_\rho(\bar{\Delta})=0$. 
		\end{enumerate}
	\end{lemma}
	\begin{proof}
		\begin{enumerate}
			\item Let $\rho$ be a supercuspidal representation with central character $\omega$. 
			As topological groups 
			$$Z(\GL_n(F))\cong F^\times \cong \la \varpi\ra \times \mc{O}_F^\times.$$
			Thus we have a factorization $\omega\cong \nu^{x+iy}\boxtimes\chi$, for some $x, y\in \br$.    
			Then, the central character of $\rho(-x)$ is isomorphic to
			$$\nu^{-x}\nu^{x+iy}\chi=\nu^{iy}\chi.$$
			As $\mc{O}_F^\times$ is compact, $\chi$ is unitary, and since $\nu^{iy}$ is unitary, we conclude that $\rho(-x)$ has unitary central character, so we take $x_\rho:=x$.  
			By \cite{Tad}*{Proposition 2.3} a supercuspidal representation is unitary if and only if it has unitary central character,
			which concludes the result. 
			\item Given a segment $\Delta$, we can write $\Delta=[\rho(a), \rho(b)]$ for unitary $\rho$ by the result above.  
			Defining $\bar{\Delta}$ as in the statement, by Lemma \ref{lem: twist_seg}
			\begin{align*}
				Q([\rho(a), \rho(b)])&\cong \nu^{(a+b)/2}\ten Q([\rho((a-b)/2), \rho((b-a)/2)])\\
				&\cong \nu^{ m_\rho(\Delta)}\ten Q(\bar{\Delta}).
			\end{align*}
			Meanwhile,  
			$$\rho\lb \lb\frac{a-b}{2} + \frac{b-a}{2}\rb/2 \rb\cong \rho,$$
			is unitary, and therefore $Q(\bar{\Delta})$ is square-integrable. 
			\item Follows directly from 2). 
		\end{enumerate}
	\end{proof}
	
	Define the character 
	$$\chi_i:M_{\vec{n}}\to \bc^\times,$$
	by sending the block diagonal matrix $\tx{diag}(A_1, \ldots, A_r)$ to $|\det(A_i)|_F$.  
	Then, 
	\begin{align}
		\lset \frac{1}{n_1}\chi_1, \ldots, \frac{1}{n_r}\chi_r\rset,\label{eqn: basis}
	\end{align}
	is a basis for $\mf{a}_{P_{\vec{n}}}^\ast$.  
	Moreover, for each $\vec{n}$, and standard parabolic $P_{\vec{n}}$ we can write each element of $\mf{a}_P^\ast$ with respect to the basis of Equation \ref{eqn: basis} as $\vec{z}=(z_1, \ldots, z_r)$.   
	As a shorthand, we will write
	$$e^{\vec{z}}\cong \nu^{z_1}\boxtimes \cdots \boxtimes \nu^{z_r},$$ 
	for the character of the corresponding Levi determined by $(z_1, \ldots, z_r)$.  
	We also note that in this description, 
	$$\mf{a}_P^{\ast, +}:=\{\lambda \in\mf{a}_M^\ast | \forall \alpha\in \Delta_P, \alpha^\vee(\lambda)>0\}=\{(a_1, \ldots, a_r)\in \br^r | a_i> a_{i+1}\}.$$	
	
	As a shorthand, we will write $I_{(n_1, \ldots, n_r)}$ for the normalized parabolic induction $I_{P_{(n_1, \ldots, n_r)}}^{\GL_n(F)}$. 
	
	\begin{theorem}[\cite{GH}*{Theorem 8.4.5}]\label{thm: temp}
		Every irreducible tempered representation of $\GL_n(F)$ is of the form
		$$\tau_{\vec{n}}:=I_{(n_1, \ldots, n_r)}^n\lb \boxtimes_{i=1}^r Q(\Delta_i)\rb,$$
		where $\Delta_i=[\nu^{a_i}\rho_i, \nu^{b_i}\rho_i]$ with each $Q(\Delta_i)$ square-integrable.
	\end{theorem}

	\begin{theorem}\label{thm: equiv_seg_stand}
		For every multisegment $a$ satisfying Definiton \ref{dnp}, $S_Q(a)$ is isomorphic to a standard representation $S_{Q(a)}$ as in Theorem \ref{thm: lc}. 
	\end{theorem}
	\begin{proof}
		Given a multisegment $a$, by Lemma \ref{lem: sqr}, for each $\Delta\in a$ there is a unique unitary representative $\rho_\Delta$ of its inertial support. 
		Let $m_1>m_2>\dots >m_r$ be the distinct midpoints $m_{\rho_\Delta}(\Delta)$ as $\Delta$ ranges over $a$.  
		Label $a=\{\Delta_{ij}\}$ such that
		the $\rho_{\Delta_{ij}}$-midpoint of $\Delta_{ij}$ is $m_i$.  
		Then, ordering the segments lexicographically by the indices satisfies Definition \ref{dnp}.  
		Indeed, if $(i,j) >(k, l)$ and $\Delta_{ij}, \Delta_{kl}$ have distinct inertial support, then neither precedes the other.  
		Thus, we may suppose they have the same inertial support, say with unitary representative $\rho$. 
		Therefore,
		$$m_\rho(\Delta_{ij})=m_i>m_k=m_\rho(\Delta_{ijk}),$$
		and thus by Lemma \ref{lem: can ord} $\Delta_{ij}$ does not precede $\Delta_{kl}$.  
		Hence the lexicographic ordering on $a=\{\Delta_{ij}\}$ satisfies Definition \ref{dnp}.  
		Let $n_{ij}$ be the sum of the length of $\Delta_{ij}$, define $n_i:=n_{i1}+\cdots + n_{is_i}$, and $\vec{n}=(n_1, \ldots, n_r)$. 
		Then,
		\begin{align*}
			S_Q(a)&\cong \bigtimes_{i=1}^r\bigtimes_{j=1}^{s_i}Q(\Delta_{ij})\\
			&\cong \bigtimes_{i=1}^r\bigtimes_{j=1}^{s_i}\nu^{m_i}\ten Q(\bar{\Delta}_{ij}) \\
			&\cong I_{\vec{n}}\lb (\boxtimes_{i=1}^r\nu^{m_i})\ten\lb\boxtimes_{i=1}^r\bigtimes_{j=1}^{s_i}Q(\bar{\Delta}_{ij})\rb\rb.
		\end{align*}
		Since $m_1>\cdots >m_r$, we see that 
		$$\vec{m}=(m_1, \ldots, m_r)\in \mf{a}_{P_{\vec{n}}}^{\ast, +},$$
		and $e^{\vec{m}}\cong \boxtimes_{i=1}^r\nu^{m_i}$.  
		Moreover, by Lemma \ref{lem: sqr} 3), for any $i,j,k$,
		$$m_\rho(\bar{\Delta}_{ij})=0=m_\rho(\bar{\Delta}_{ik}).$$
		Thus, for each $j,k$ the segments $\bar{\Delta}_{ij}, \bar{\Delta}_{ik}$ are not linked, hence by \cite{Zel}*{Theorem 4.2} $\bigtimes_{j=1}^{s_i}Q(\bar{\Delta}_{ij})$ is irreducible, and as each $Q(\Delta_{ij})$ is square-integrable, each
		$$\tau_i:=\boxtimes_{i=1}^r\bigtimes_{j=1}^{s_i}Q(\bar{\Delta}_{ij}),$$
		is an irreducible tempered representation of $\GL_{n_i}(F)$, and thus $\tau:=\boxtimes_{i=1}^r\tau_i$
		is an irreducible tempered representation of $M_{\vec{n}}$.  
		In other words, $(P_{\vec{n}}, e^{\vec{m}}, \tau)$ is a Langlands triple.  
		Therefore,
		$$S_Q(a)\cong I_{\vec{n}}\lb (\boxtimes_{i=1}^r\nu^{m_i})\ten\lb\boxtimes_{i=1}^r\bigtimes_{j=1}^{s_i}Q(\bar{\Delta}_{ij})\rb\rb\cong I_{\vec{n}}(e^{\vec{m}}\ten \tau),$$
		is a standard representation with unique irreducible quotient $Q(a)$, and thus $S_Q(a)\cong S_{Q(a)}$. 
	\end{proof}

	\subsection{Relating $S_Z(a)$ and $S_Q(a)$}
	
	In this section, we prove Proposition \ref{prop: zmult=qmult}, which states that 
	$$[S_Q(a):Q(b)]=[S_Z(a):Z(b)].$$
	The version of the $p$-adic Kazhdan-Lusztig hypothesis \cite{Zel}*{Hypothesis 1.9} involves the representations $S_Z(a)$, while the generalized version \cite{Vog}*{Conjecture 8.11} is phrased in terms of standard representations as defined by Definition \ref{def: lc}.  By Theorem \ref{thm: equiv_seg_stand}, the representations $S_Q(a)$ are exactly the standard representations of Definition \ref{def: lc}, and therefore Proposition \ref{prop: zmult=qmult} allows us to relate \cite{Zel}*{Hypothesis 1.9} and \cite{Vog}*{Conjecture 8.11}.

	Before proving the main results of this section, we recall some facts about the covariant Aubert-Zelevinsky duality functor $D$, which was originally introduced in \cite{ZelI2} and later generalized in \cite{Aub}. 
	In unpublished notes, Bernstein introduced a contravariant duality functor $D'$. 
	Letting $(-)^\vee$ be the functor sending a representation to its contragradient, the work of \cite{PN} proves that $D'$ is isomorphic to $D'=D\circ (-)^\vee$.  
	
	Suppose $a=\{\Delta_1, \ldots, \Delta_r\}$ is a multisegment (always assumed to follow Definition \ref{dnp}), let $n_i$ be the length of $\Delta_i$, and let $D_i$ be Aubert-duality on $\GL_{n_i}(F)$.  Define $\vec{n}^\vee:=(n_r, \ldots, n_1)$.
	Then, letting $s$ be the longest element of the Weyl group of $\GL_n(F)$,
	$$\bar{P}_{\vec{n}}=P_{\vec{n}^\vee}^s.$$
	Thus, we compute
	\begin{align*}
		D(S_Z(a))&=D \lb Z(\Delta_1)\times\cdots \times Z(\Delta_r)\rb\\
		&\cong D \circ I_{P_{\vec{n}}}\lb Z(\Delta_1)\boxtimes\cdots \boxtimes Z(\Delta_r)\rb\\
		&\cong I_{\bar{P}_{\vec{n}}}^G\lb D_1Z(\Delta_1))\boxtimes\cdots \boxtimes D_r(Z(\Delta_r))\rb\\
		&\cong I_{P_{\vec{n}^\vee}^s}^G\lb Q(\Delta_1)\boxtimes\cdots \boxtimes Q(\Delta_r)\rb\\
		&\cong I_{P_{\vec{n}^\vee}}\lb (Q(\Delta_1)\boxtimes\cdots \boxtimes Q(\Delta_r))^{s^{-1}}\rb\\
		&\cong I_{P_{\vec{n}^\vee}}\lb Q(\Delta_r)\boxtimes\cdots \boxtimes Q(\Delta_1)\rb.
	\end{align*}
	By \cite{ZelI2}*{Theorem 1.9} "$\times$" defines a commutative product on the Grothendieck group $K\Rep(\GL_n(F))$ of smooth representations, and thus, writing $[\pi]$ for the class of a representation in the Grothendieck group
	\begin{align*}	
		[D(S_Z(a))]&=[I_{P_{\vec{n}^\vee}}\lb Q(\Delta_r)\boxtimes\cdots \boxtimes Q(\Delta_1)\rb]\\
		&= [Q(\Delta_r)\times\cdots \times Q(\Delta_1)]\\
		&= [Q(\Delta_1)\times\cdots\times  Q(\Delta_r)]\\
		&=[S_Q(a)]. 
	\end{align*}

	By \cite{Rod}*{Theorem 7} $D(Z(a))\cong Q(a)$, and therefore $D$ induces $\bz$-linear isomorphism of Grothendieck groups, 
	\begin{align*}
		K\Rep(G)&\xrightarrow{D} K\Rep(G)\\
		[Z(a)]&\mapsto [Q(a)]\\
		[S_Z(a)]&\to [S_Q(a)]
	\end{align*}
	
	Let $m_Z(b;a)$ denote the multiplicity of $Z(b)$ in $S_Z(a)$ and $m_Q(b;a)$ denote the multiplicity of $Q(b)$ in $S_Q(a)$.  
	
	\begin{proposition}\label{prop: zmult=qmult}
		For any multisegments $a, b$ we have $m_Q(b; a)=m_Z(b;a)$. 
	\end{proposition}
	\begin{proof}
		Applying $D$ to 
		$$[S_Z(a)]=\sum_{b\leq a}m_Z(b;a)[Z(b)],$$
		yields
		\begin{align*}
			[S_Q(a)]&=\sum_{b\leq a}m_Z(b;a)[Q(a)],
		\end{align*}
		but since 
		$$[S_Q(a)]=\sum_{b\leq a}m_Q(b;a)[Q(a)],$$
		and the $[Q(b)]$ are a basis for $K\Rep(\GL_n(F))$, we conclude that
		$$[S_Q(a):Q(b)]=m_Q(b;a)=m_Z(b;a)=[S_Z(a):Z(b)].$$
	\end{proof}
	
	Thus, throughout the rest of the article, we will simply write $m(b;a)$ for $m_Z(b;a)=m_Q(b;a)$.


	\subsection{Reduction to multisegments of simple inertial support}
	
	The main result of this section is Proposition \ref{prop: prodmult}, which allows one to decompose general multisegments in terms of \emph{simple} multisegments (Definition \ref{def: simp}), and will be an essential step in proving the general case of the $p$-adic Kazhdan-Lusztig hypothesis. 
	First, we require a technical result.

	\begin{lemma}\label{lem: multirays}
		Let $a_1, \ldots, a_r$ be multisegments such that if $\Delta\in a_i, \Delta'\in a_j$ are linked, then $i=j$.  
		Then, the multisegments $b\leq a$ are precisely $b=b_1+\cdots +b_r$ where $b_i\leq a_i$.  
		Moreover, for each such $b$, if $\Delta\in b_i, \Delta'\in b_j$ are linked, then $i=j$. 
	\end{lemma}
	{\allowdisplaybreaks
		\begin{proof}
			First, for any $1\leq k \leq r$, let $b_k$ be obtained by a single simple operation on $\Delta, \Delta'\in a_k$.  
			Then $\Delta, \Delta'\in a$, and
			$$b=a_1+\cdots + a_{k-1}+b_k+a_{k+1}+\cdots + a_r,$$
			is obtained by the simple operation on $\Delta, \Delta'$. 
			By induction, it follows that choosing any $b_i\leq a_i$ for each $i$, we have 
			$$(b_1+\cdots + b_r)\leq (a_1+\cdots + a_r).$$
			Note that this direction did not actually require the condition in the lemma statement about the segments being (not) linked. 
			
			Now, suppose that $b$ is obtained by a simple operation from $a$, on linked segments $\Delta$ and $\Delta'$.  
			By assumption, there is some $k$ for which $\Delta, \Delta'\in a_k$.    
			Therefore, letting $b_k$ be the multisegment obtained from $a_k$ from the simple operation on $\Delta$ and $\Delta'$, we find that 
			$$b=a_1+\cdots +a_{k-1}+b_k+a_{k+1}+\cdots + a_r.$$
			
			It follows that for any segment $b\leq a$, there exists $b_1\leq a_1, \ldots, b_s\leq a_s$ such that 
			$$b=b_1+\cdots + b_s.$$
			This concludes the first part of the statement.  
			
			As with the previous statement it suffices to prove the result when $b$ is obtained from $a$ from a single multisegment operation, the general case then follows by induction. 
			We know from the above that for some $i\in\{1, \ldots, r\}$,
			$$b=b_1+\cdots + b_r,$$
			where $b_i=a_i$ for $i\neq k$, and $b_k$ is obtained from a single simple operation on $a_k$, say on segments $\Delta_1=[x, y], \Delta_2=[z, w]\in a_k$. 
			
			We wish to show that if $\Delta=[u, v], \Delta'=[r, s]\in b$ are linked, then they belong to the same $b_j$.  
			
			If neither of $\Delta, \Delta'$ are in $b_k$, then for some $i, j\neq k, \Delta\in b_i=a_i, \Delta'\in b_j=a_j$, and thus by our assumption on $a$, $i=j$, hence $\Delta, \Delta'\in b_i$.  
			If both $\Delta$ and $\Delta'$ belong to $b_k$, there is nothing to show.
			Thus it remains to show by contradiction that it can not be that one, say $\Delta$, belongs to $\Delta\in b_k$, while $\Delta'\in b_i=a_i$ for $i\neq k$. 
			
			There are 3 cases: $\Delta\in a_k\backslash\{\Delta_1, \Delta_2\}$, or $\Delta=\Delta_1\cup\Delta_2$, or $\Delta=\Delta_1\cap \Delta_2$.
			{\allowdisplaybreaks
				\begin{enumerate}
					\item If $\Delta\in a_k\backslash\{\Delta_1, \Delta_2\}$ then $\Delta\in a_k$ and since $\Delta'\in b_j=a_j$ by the assumption on $a$ it must be that $j=k$, contrary to our assumptions.  
					\item Assume without loss of generality that $\Delta_1$ precedes $\Delta_2$. 
					If $[u, v]=\Delta=\Delta_1\cup\Delta_2=[x, w]$ is linked with $\Delta'=[r, s]$, then,
					($r<x$ and $x-1\leq s<w$) or ($w<s$ and $x<r\leq w+1$).
					We will assume the former, the case of the latter being entirely similar.  
					
					If $x<y$, then $\Delta'=[r, s]\in b_j, a_j$ and $\Delta_1=[x, y]\in a_k$ are linked, implying that $j=k$, contrary to our assumptions.  
					
					If $y\leq s$, then, since $\Delta_1=[x, y], \Delta_2[z, w]$ are linked, and $\Delta_1$ precedes $\Delta_2$, we must have $z<y\leq s<w$.  
					Therefore $\Delta'$ is linked with $\Delta_2$, which is again a contradiction. 

					\item If $\Delta=\Delta_1\cap \Delta_2$, the argument is entirely similar.  
			\end{enumerate}}
	\end{proof}}
	
	{\allowdisplaybreaks
		\begin{proposition}\label{prop: prodmult}
			Consider a segment $a=a_1+\cdots + a_r$ such that if $\Delta\in a_i, \Delta'\in a_j$ are linked, then $i=j$.  
			For any segments $b_i\leq a_i$,
			$$m(b_1+\cdots +b_r;a)=m(a_1; b_1) \cdots m(a_r; b_r).$$
	\end{proposition}}
	{\allowdisplaybreaks
		\begin{proof}
			By \cite{ZelI2}*{Proposition 8.5}, if $a_1, \ldots, a_r$ are multisegments such that $\Delta\in a_i$ and $\Delta'\in a_j$ being linked implies $i=j$, then for $a=a_1+\cdots + a_r$,
			$$Z(a)=Z(a_1+\cdots +a_r)\cong Z(a_1)\times \cdots \times Z(a_r).$$
			Labeling the segments $a_i=\{\Delta_{i1}, \ldots, \Delta_{i,s_i}\}$, we see that in the Grothendieck group 
			{\allowdisplaybreaks
				\begin{align*}
					\sum_{b\leq a} m(b;a)[Z(b)]&=[S_Z(a)] \\
					&=\left[\bigtimes_{i=1}^r\bigtimes_{j=1}^{s_i}Z(\Delta_{ij})\right]\\
					&=\bigtimes_{i=1}^r\left[\bigtimes_{j=1}^{s_i}Z(\Delta_{ij})\right]\\
					&=\bigtimes_{i=1}^r[S_Z(a_i)]\\
					&=\bigtimes_{i=1}^r\sum_{\substack{(b_1, \ldots, b_r) \\ b_i\leq a_i}}m(b_i; a_i)[Z(b_i)]\\
					&=\sum_{\substack{(b_1, \ldots, b_r) \\ b_i\leq a_i}} m(a_1; b_1)\cdots  m(a_r; b_r) [Z(b_1)\times\cdots \times Z(b_r)].  
			\end{align*}}
			
			By Lemma \ref{lem: multirays}, we know that for one of the $(b_1, \ldots, b_r)$ where each $b_i\leq a_i$, if there existed some $\Delta \in b_i, \Delta' \in b_j$ which were linked, it must be that $i=j$.  
			Thus, by \cite{ZelI2}*{Proposition 8.5} 
			$$Z(b_1)\times\cdots \times Z(b_r)\cong Z(b_1+\cdots + b_r).$$
			Therefore, the above becomes,
			\begin{align*}
				\sum_{b\leq a} m(b;a)[Z(b)]
				&=\sum_{\substack{(b_1, \ldots, b_r) \\ \forall i, b_i\leq a_i}} m(a_1; b_1)\cdots  m(a_r; b_r) [Z(b_1)\times\cdots \times Z(b_r)]\\
				&=\sum_{\substack{(b_1, \ldots, b_r)\\
						\forall i, b_i\leq a_i}} m(a_1; b_1) \cdots  m(a_r; b_1)[Z(b_1+\cdots + b_r)].
			\end{align*}
			By Lemma \ref{lem: multirays} we know each $b\leq a$ is of the form $(b_1, \ldots, b_r)$ for $b_i\leq a_i$, and thus re-writing the sum on the left,
			\begin{align*}
				\sum_{\substack{(b_1, \ldots, b_r)\\ \forall i, b_i\leq a_i}}m(b_1+\cdots +b_r; a)[Z(b_1+\cdots+b_r)]&=\sum_{\substack{(b_1, \ldots, b_r)\\
						\forall i, b_i\leq a_i}} m(a_1; b_1) \cdots  m(a_r; b_1)[Z(b_1+\cdots + b_r)],
			\end{align*}
			and since the $Z(b_1+\cdots +b_r)$ are basis elements in the Grothendieck group, and each such element appears exactly once on each side of the above equality, we conclude that for all choices of $b_i\leq a_i$,
			$$m(b_1+\cdots +b_r; a_1+\cdots + a_r)=m(a_1; b_1)\cdots m(a_r; b_r). $$
	\end{proof}}


	\subsection{Reduction to simple inertial support}

	The main result of this section is Theorem \ref{thm: simpunramseg}, which is a key step in an essential result for this article being Corollary \ref{cor: unramseg}: for multisegments $a$ and $b$, the value $m(a;b)$ does not depend on the inertial supports. 
	Undoubtedly, Corollary \ref{cor: unramseg} is known to experts, but it appears the proof is only written down in the pre-print \cite{PyvPre}, and did not appear in the published version \cite{Pyv}.  
	Thus, for the sake of the completeness of the published literature, we offer another argument here.

	\begin{definition}
		Given an open compact subgroup $J$ of $G$, and an irreducible representation $(W, \tau)$ in $\Rep(J)$, we define $\mc{H}(G, \tau)$ to be the algebra of compactly supported functions $\psi:G\to \End_G(W^\vee)$, such that for all $k_1, k_2\in J$, and $g\in G$, 
		$$\psi(k_1gk_2)=\tau^\vee(k_1)\circ \psi(g) \circ \tau^\vee(k_2),$$
		and whose product is given by convolution
		$$(\varphi\ast \psi) (x)=\int_G\varphi(g)\psi(g^{-1}x)dg.$$
	\end{definition}
	
	Given $\psi\in \mc{H}(G, \tau)$, we define
	$$\check{\psi}(g):=\psi(g^{-1})^\vee.$$
	
	Then, $\check{\psi}\in \mc{H}(G, \tau^\vee)$, and the map
	\begin{align*}
		\mc{H}(G, \tau)&\to \mc{H}(G, \tau^\vee) \\
		\psi&\mapsto \check{\psi}
	\end{align*}
	defines an anti-isomorphism.  
	
	A case of particular importance for us will be the pair $(\mc{I}, 1_{\mc{I}})$ where $\mc{I}$ is that Iwahori subgroup of $\GL_n(F)$, and $1_{\mc{I}}$ the trivial representation of $\mc{I}$.  
	In this case, we call $\mc{H}(G, 1_{\mc{I}})$ the \emph{Iwahori-Hecke algebra}.
	Since $1_{\mc{I}}^\vee\cong 1_{\mc{I}}$ we also have an isomorphism $\mc{H}(G, 1_{\mc{I}})\cong \mc{H}(G, 1_{\mc{I}}^\vee)$.  
	
	Given a collection of pairs $(J_i, \tau_i)$ where $J_i$ is a compact subgroup of $\GL_{n_i}(F)$, $\prod_{i=1}^r J_i$ is a subgroup of 
	$$M_{n_1, \ldots, n_r}=\GL_{n_1}(F)\times \cdots \times \GL_{n_r}(F).$$
	Every smooth function $f:G_{n_1, \ldots, n_r}\to \bc$ can be factored as a product of functions on the individual factors, and thus we obtain an isomorphism by multiplication 
	$$\mc{H}(\GL_{n_1}(F), \tau_1)\ten_\bc\cdots \ten_\bc \mc{H}(\GL_{n_r}(F), \tau_r)\to \mc{H}(M, \boxtimes_{i=1}^r\tau_i).$$
	
	Recall that every smooth irreducible representation $\pi$ of $G$, there exist a supercuspidal $\sigma$ of a Levi subgroup $M$, and a parabolic $P\supset M$ such that $\pi$ is a subquotient of $I_{M}^G(\sigma)$.  
	The pair $(M, \sigma)$ is defined up to $G$-conjugation, and the equivalence class $(M, \sigma)_G$ under conjugation is called the \emph{cuspidal support} of $\pi$. 
	
	\begin{definition}\label{def: simp}
		We call a cuspidal support of the form 
		$$(M_{dn_1, \ldots, dn_r}(F), \nu^{a_1}\rho\boxtimes \cdots \boxtimes\nu^{a_r}\rho),$$
		\emph{simple}.  
		Likewise, if $a=\{\Delta_1, \ldots, \Delta_r\}$ is a multisegment for which there exists a supercusidal representations $\rho$ and complex numbers $a_i, b_i\in \bc$ such that $\Delta_i=[\rho(a_i), \rho(b_i)]$, then we say the inertial support of $a$ is simple. 
	\end{definition}
	
	Since a simple cuspidal support is defined up to conjugation, we may assume that for each $i\in\{1, \ldots, r\}$ we have $\Re(a_i)\geq \Re(a_{i+1})$.  
	Thus the multisegment
	$$a:=\{[\rho(a_1), \rho(a_1)], \ldots, [\rho(a_r), \rho(a_r)]\},$$
	satisfies Definition \ref{dnp}.  
	Therefore, if an irreducible representation $\pi$ has cuspidal support $(M_{dn_1, \ldots, dn_r}, \nu^{a_1}\rho\boxtimes \cdots \boxtimes\nu^{a_r}\rho),$ there must exist a multisegment $b\leq a$ such that $\pi\cong Q(b)$.

	Thus far, our notation for the norm character $\nu:\GL_1(F)\to \bc^\times$, and segments $\Delta=[c, d]=[\nu^c, \nu^d]$ with trivial inertial support has suppressed the underlying field $F$ as it should always be understood from context.
	We will see that ultimately, the multiplicities do not depend on the base field in Corollary \ref{cor: unramseg}, but in order to rigorously explain why this is so, we must use notation which highlights the base field.  
	Therefore, for the purposes of the following theorem, and some results in later sections, for a $p$-adic field $F$, we will write $\nu_F:\GL_1(F)\to\bc^\times$ for the norm-character. 
	
	{\allowdisplaybreaks
		\begin{theorem}\label{thm: simpunramseg}
			Given multisegments $a, b$ with the the same simple inertial support, represented by a supercuspidal representation $\rho$ of $\GL_d(F)$, we can write
			\begin{align*}
				a&=\{[\nu_F^{a_i}\rho, \nu_F^{b_i}\rho]\}_{i=1}^r,\\
				b&=\{[\nu_F^{c_i}\rho, \nu_F^{d_i}\rho]\}_{i=1}^s,
			\end{align*}
			of $\GL_{nd}(F)$,
			there exists a finite extension $E/F$ such that for the multisegments
			\begin{align*}
				a^\circ&=\{[\nu^{a_i}_E,\nu^{b_i}_E]\}_{i=1}^r\\
				b^\circ&=\{[\nu^{c_i}_E, \nu^{d_i}_E]\}_{i=1}^s,
			\end{align*}
			of $\GL_n(E)$, such that
			$$m(a;b)=m(a^\circ; b^\circ).$$
	\end{theorem}}
	{\allowdisplaybreaks
		\begin{proof}
			In \cite{BK}*{Section 7.4}, a (family of) equivalence(s) of categories
			$$M_\tau:\Rep_\tau(\GL_{nd}(F))\to \Mod\lb\mc{H}(G, \tau)\rb$$
			is described, where $\tau$ is an irreducible representation of a compact open subgroup $J$, such that $(J, \tau)$ is a \emph{simple type} (see \cite{BK}*{5.5.10}), and $\Rep_\tau(\GL_{nd}(F)))$ is a full subcategory of $\Rep(\GL_{nd}(F)$ of representations "having type" $\tau$. 
			By \cite{BK}*{Theorem 8.4.3} there exists a supercuspidal representation $\rho$ of $\GL_d(F)$ such that the irreducible representations with cuspidal support
			$$(M, \sigma):=(\GL_d(F)^n, \nu^{a_1}\rho\boxtimes \cdots \boxtimes\nu^{a_r}\rho),$$
			are exactly those with type $(J,\tau)$.  
			In other words, we can forgo the definition of $\Rep_\tau(G)$ here, since it is the full subcategory of representations of $\Rep(G)$, where the irreducible subquotients have cuspidal support $(M, \sigma)$. 
			
			Moreover, writing $G=\GL_{nd}(F)$, the results of \cite{BK}*{Section 7.4} demonstrate that there exist a finite extension $E/F$, a subgroup $C^\times\cong \GL_n(E)$ of $G$, with an Iwahori subgroup $\mc{I}$, and an isomorphism   
			$$\Psi^n:\mc{H}(C^\times, \1_{\mc{I}})\to \mc{H}(\GL_{nd}(F), \tau).$$
			
			The main result of \cite{Pro}*{Corollary 6.27} is the following:
			
			Let $M\cong M_{n_1d, \ldots, n_rd}$ be a Levi subgroup with appropriately chosen simple types $(J_i, \tau_i)$. Then, there is a Levi subgroup $L=L_{n_1, \ldots, n_r}$,  an inclusion an inclusion $i:\mc{H}(L, \1_{\mc{I}\cap L})\to \mc{H}(C^\times, \1_{\mc{I}})$, and isomorphisms $\alpha:\mb{H}_{n, q'}\to \mc{H}(C^\times, \1_{\mc{I}})$ and $\beta:\mb{H}_{\vec{n}, q'}\to \mc{H}(L, \1_{\mc{I}\cap L})$ such that for the extension of scalars functor 
			$$\begin{tikzcd}
				\Rep_\tau(\GL_{nd}(F)) \arrow[r, "\Psi^n_\ast\circ M_\tau"] & \Mod(\mc{H}(C^\times, \1_{\mc{I}})) \\
				\Rep_{\boxtimes \tau_i}(M) \arrow[u, "I_P"] \arrow[r, "(\ten_{i=1}^r \Psi^{n_i})_\ast\circ M_{\tau_M}"'] & \Mod(\mc{H}(L, \1_{\mc{I}\cap L})) \arrow[u, "i_\ast"'] 
			\end{tikzcd}$$
			and
			$$(\ten_{i=1}^r \Psi_{n_i})_\ast\circ M_{\tau_M}\cong \bigotimes_{i=1}^r(\Psi_{n_i}^\ast\circ M_{\tau_i}).$$
			Thus, writing
			\begin{align*}
				F_\tau&:= \Psi^n_\ast\circ M_\tau:\Rep_\tau(\GL_{nd}(F))\to \Mod(\mc{H}(C^\times, \1_{\mc{I}}))\\
				F_{\tau_M}&= (\ten_{i=1}^r \Psi^{n_i})_\ast\circ M_{\tau_M}:\Rep_{\boxtimes \tau_i}(M)\to \Mod(\mc{H}(L, \1_{\mc{I}\cap L})), 
			\end{align*}
			we have that
			$$F_{\tau_M}(\pi_1\boxtimes\cdots \boxtimes\pi_r)\cong F_{\tau_1}(\pi_1)\ten_\bc \cdots \ten_\bc F_{\tau_r}(\pi_r).$$
			If $\tau$ is a simple type as above, then there is a type $\sigma$ such that in the above $\tau_i=\sigma$ for all $i$.   
			
			For each $\mc{H}(G, \tau)$-module $M$, and every complex number $c\in \bc$, \cite{BK}*{Section 7.5} defines a $\mc{H}(G, \tau)$-module $M(c)$ such that, by \cite{BK}*{Proposition 7.5.12},
			$$\Psi(\pi(c))\cong \Psi(\pi)(c),$$
			recalling that we previously defined $\pi(c):=\nu^c\ten\pi$ for representations.

			For a segment $\Delta=[\rho(a), \rho(b)]$ of $\GL_{nd}(F)$ with simple type $(J, \tau)$, define $\Delta^\circ:=[\nu_E^a, \nu_E^b]$, which has type $(\mc{I}_n, \1_{\mc{I}_n})$  where $\mc{I}_n$ is the Iwahori subgroup of $\GL_n(E)$. 
			We can choose a type $\sigma$ of $\GL_d(F)$ such that in the above commuting diagram each $\tau_i=\sigma$.  
			Then,
			\begin{align*}
				F_\tau(S_Z(\Delta)_\rho)&=F_\tau \circ I \lb \rho(a) \boxtimes \cdots \boxtimes \rho(b) \rb\\
				&=i_\ast \lb F_\sigma(\rho(a))\boxtimes\cdots\boxtimes  F_\sigma(\rho(b))\rb\\
				&=i_\ast \lb F_\sigma(\rho(a))\boxtimes\cdots\boxtimes  F_\sigma(\rho(b))\rb\\
				&\cong i_\ast \lb F_\sigma(\rho)(a)\boxtimes\cdots\boxtimes  F_\sigma(\rho)(b)\rb.
			\end{align*}

			In this case, $L$ is a torus, and $F_\sigma(\rho)$ is a simple  module over $\mc{H}(L, \1_{\mc{I}\cap L})\cong \mc{H}(\GL_1(E), \GL_1(\mc{O}_E))$, and thus every simple module, namely $F_\sigma(\rho)$,  must be of the form
			$$F_1(\nu_E^c)\cong F_1(1)(c),$$
			for some $c\in \bc$.
			
			For the Iwahori subgroup $\mc{I}_{n, E}$ of $\GL_n(E)$, we have the type $(\mc{I}_{n, E}, \1)$, and the functor
			$$F_n:=F_{(\mc{I}_{n, E}, \1)}.$$
			Similarly, for a Levi $M_{\vec{n}}$ of $\GL_n(E)$, we write
			$$F_{\vec{n}}:=F_{(M_{\vec{n}}, \mc{I}_{M_{\vec{n}}})}.$$ 
			Thus, above can be written as
			{\allowdisplaybreaks
				\begin{align*}
					F_\tau(S_Z(\Delta))
					&\cong i_\ast \lb F_\sigma(\rho(a))\boxtimes\cdots\boxtimes F_\sigma(\rho(b))\rb\\
					&\cong i_\ast \lb F_\sigma(\rho)(a)\boxtimes\cdots\boxtimes F_\sigma(\rho)(b)\rb\\
					&\cong i_\ast \lb F_1(1)(c)(a)\boxtimes\cdots\boxtimes  F_1(1)(c)(b)\rb\\
					&\cong i_\ast \lb F_1(\nu_E^{a+c})\boxtimes\cdots\boxtimes  F_1(\nu_E^{b+c})\rb\\
					&\cong i_\ast \circ F_{(1, \ldots, 1)} \lb \nu_E^{a+c}\boxtimes\cdots\boxtimes  \nu_E^{b+c}\rb\\
					&\cong F_n\circ I \lb \nu^{a+c}_E\boxtimes\cdots\boxtimes  \nu^{b+c}\rb\\
					&\cong F_n\circ (\nu_E^c\circ \det)\ten I \lb \nu_E^{a}\boxtimes\cdots\boxtimes  \nu_E^{b}\rb\\
					&\cong F_n(S_Z(\Delta^\circ))(c).
			\end{align*}}
			Since $Z(\Delta)$ is the unique irreducible subrepresentation of $S_Z(\Delta)$ we know that $ F_\tau(Z(\Delta))$ must be the unique irreducible representation of  $F_\tau(S_Z(\Delta^\circ))\cong F_n(S_Z(\Delta^\circ))(c)$, which has unique irreducible subrepresentation $F_n(Z(\Delta^\circ)(c))$, and therefore
			$$ F_\tau(Z(\Delta))\cong F_n(Z(\Delta^\circ))(c)\cong F_n(Z(\Delta^\circ)(c)).$$

			Then, 
			\begin{align*}
				F_\tau(S_Z(a))&\cong F_\tau\circ I_{d\vec{n}}\lb Z(\Delta_1)\boxtimes\cdots \boxtimes Z(\Delta_r) \rb\\
				&\cong i_\ast \circ F_{\tau_M}\lb Z(\Delta_1)\boxtimes\cdots \boxtimes  Z(\Delta_r)\rb\\
				&\cong i_\ast \lb F_{\tau_1}(Z(\Delta_1))\boxtimes \cdots \boxtimes F_{\tau_r}(Z(\Delta_r))\rb\\
				&\cong i_\ast\lb  F_{n_1}(Z(\Delta_1^\circ )(c))\boxtimes \cdots \boxtimes  F_{n_r}(Z(\Delta_r^\circ)(c))\rb\\
				&\cong i_\ast \circ F_{(n_1, \ldots, n_r)}(Z(\Delta_1^\circ)(c)\boxtimes\cdots \boxtimes Z(\Delta_r^\circ)(c))\\
				&\cong F_n\circ I_{\vec{n}}(Z(\Delta_1^\circ)(c)\boxtimes\cdots \boxtimes Z(\Delta_r^\circ)(c))\\
				&\cong F_n( I(Z(\Delta_1^\circ)\boxtimes\cdots \boxtimes Z(\Delta_r^\circ)))(c)\\
				&\cong F_n(S_Z(a^\circ)(c)).
			\end{align*}
			Again, since $Z(a)$ is the unique irreducible representation $S_Z(a)$, it must be that $F_\tau(Z(a))$ is isomorphic to the unique irreducible subrepresentation of $F_\tau(S_Z(a))\cong F_n(S_Z(a^\circ)(c))$, which is $F_n(Z(a^\circ)(c))$.  
			Therefore,
			\begin{align*}
				m(a;b)&=[S_Z(a): Z(b)] \\
				&=[F_\tau(S_Z(a)):F_\tau(Z(b))]\\
				&=[F_n(S_Z(a^\circ)(c)): F_n(Z(b^\circ)(c))] \\
				&=[S_Z(a^\circ)(c):Z(b^\circ)(c)]\\
				&=[S_Z(a^\circ):Z(b^\circ)], \ \ \ \ \tx{ by Lemma \ref{lem: twist_seg}}\\
				&=m(a^\circ;b^\circ). 
			\end{align*}
	\end{proof}}
	
	The above demonstrates that the multiplicities $m(a;b)$ are entirely determined by the case of multisegments with trivial inertial support. 
	While it is generally understood, that for $p$-adic fields $F$ and $E$, and multisegments
	\begin{align*}
		a_F&:= \{[\nu_F^{a_i}, \nu_F^{b_i}]\}_{i=1}^r\\
		a_E&:= \{[\nu_E^{a_i}, \nu_E^{b_i}]\}_{i=1}^r\\
		b_F&:= \{[\nu_F^{c_i}, \nu_F^{d_i}]\}_{i=1}^s\\
		b_E&:= \{[\nu_E^{c_i}, \nu_E^{d_i}]\}_{i=1}^s\\
	\end{align*}
	it should be the case that 
	$$m(a_F;b_F)=m(a_E;b_E),$$
	though this has never been precisely articulated.  
	We will see that indeed this holds in Corollary \ref{cor: unramseg}.


	\subsection{Standard representations as modules over affine Hecke algebras}\label{ssec: heckmod}
	
	In this section, we prove Theorem \ref{thm: heckemod} which offers an explicit description of modules over the algebra $\mb{H}_{n, q}$ corresponding to the representations $S_Z(a)$, where $a$ has trivial inertial support.  
	Theorem \ref{thm: heckemod} will be crucial to relate the result of \cite{CG}*{Theorem 8.6.23}, to \cite{Zel}*{Hypothesis 1.9} by way of \cite{Ari}*{Theorem 3.2}. 
	
	First, we describe how each representation $(V, \pi)$ of $\GL_n(F)$ with an $\mc{I}$-fixed vector determines an $\mc{H}(G, \mc{I})$-module. 
	First, we consider that $\mc{I}$-fixed vectors $V^{\mc{I}}$, and define an action of $f\in \mc{H}(\GL_n(F), \mc{I})$ on $v\in V^{\mc{I}}$ by 
	$$f\cdot v:=\int_{\GL_n(F)}f(g)\pi(g)vdg.$$
	The functor sending $(V, \pi)$ to the $\mc{H}(\GL_n(F), \mc{I})$-module $V^{\mc{I}}$ determines an equivalence of categories.
	
	To this end, we must first recall some facts about the structure of the algebra $\mc{H}(\GL_n(F), \mc{I})$.  
	Choosing a normalization for the Haar measure for which $\tx{vol}(\mc{I})=1$ for each $w\in W(\GL_n(F))$ and each dominant cocharacter $\mu\in X_\ast(T)^{\tx{dom}}$ of a fixed maximal torus $T$ in $\GL_n(F)$, we define the elements
	\begin{align*}
		\tilde{T}_w&:=\tx{ch}(\mc{I}w\mc{I}),\\
		\tilde{T}_\mu&:=\tx{ch}(\mc{I}\mu(\varpi)\mc{I}),
	\end{align*}
	of $\mc{H}(G, \mc{I})$, with $\tilde{T}_\mu$ invertible.  
	Writing $\mu=\mu_1-\mu_2$ as a difference of dominat cocharacters, we define
	$$\tilde{T}_\mu:=\tilde{T}_{\mu_1}\tilde{T}_{\mu_2}^{-1}.$$
	
	If $w_i$ is the element corresponding to the transposition $(i, i+1)$, then we define
	\begin{align*}
		S_i&:=\tilde{T}_{w_i}, \\
		X_j&:=\tilde{T}_{\varepsilon_j},
	\end{align*}	
	with relations
	\begin{alignat*}{2}
		(S_i+1)(S_i-q)&=0 && 1\leq i \leq n-1,\\
		S_iS_{i+1}S_i&=S_{i+1}S_iS_{i+1}, \ \ \ \ && 1\leq i\leq n-2\\
		S_iS_j&=S_jS_i, && |i-j|\geq 2 \\
		X_iX_j&=X_jX_i, && i \neq j\\
		X_jS_i&=S_iX_j, && i\neq j, j-1 \\
		S_iX_{i+1}S_i&=X_i, && 1\leq i \leq n-2
	\end{alignat*}
	
	Then, by \cite{Pro}*{Equations 2.10} that the $T_i$, $X_i$ generate $\mc{H}(\GL_n(F), \mc{I})$. 
	
	Noting that when \cite{Ari} writes "$q$", he means what we would call here "$\sqrt{q}$".  
	With this in mind, we will write $\mb{H}_{n, q}$ for what \cite{Ari} calls $\mc{H}_{q^2}$.  
	Let $\hat{G}=\GL_n(\bc)$, take $\hat{T}$ to be a maximal torus given by the diagonal matrices, and $\hat{W}$ the Weyl group. 
	Write $w_i$ for the simple reflection of $\hat{W}$ corresponding to $(i, i+1)$, let $\omega_i\in X^\ast(\hat{T})$ be the character such that
	$$\omega_i(\tx{diag}(t_1, \ldots, t_n))=t_i,$$
	and set $\alpha_i=\omega_i-\omega_{i+1}$. 
	Then, the algebra $\mb{H}_{n, q}$ is generated by $T_1, \ldots, T_{n-1}$, and $\theta_x$ for $x\in X^\ast(\hat{T})$, subject to the relations
	\begin{alignat*}{2}
		(T_i-q)(T_i+q^{-1})&=0, && 1\leq i \leq n-1\\
		T_iT_{i+1}T_i&=T_{i+1}T_iT_{i+1}, && 1\leq i \leq n-2\\
		T_iT_j&=T_jT_i, && j\geq i+2 \\
		\theta_x\theta_y&=\theta_y\theta_x, && x, y\in X^\ast(T)\\
		T_i\theta_x&=\theta_x T_i, && w_i(x)=x \\
		T_i\theta_xT_i &=\theta_{w_i(x)}, && w_i(x)=x+\alpha_i.
	\end{alignat*}
	
	We define an isomorphism
	$$\beta^{n, q}:\mb{H}_{n, q}\to \mc{H}(\GL_n(F), \mc{I}),$$
	by
	\begin{align*}
		\beta^{n, q}(\theta_{\omega_i})&=q^{i-\frac{n+1}{2}}X_i \\
		\beta^{n, q}(T_i)&=\sqrt{q}^{-1}S_i.
	\end{align*}
	
	Since the $S_i, X_i$ generate $\mc{H}(\GL_n(F), \mc{I})$, so will any scalar multiples of these elements.  
	To verify that this respects the relations in either algebra,
	\begin{align*}
		\beta^{n, q}\lb(T_i-\sqrt{q})(T_i+\sqrt{q})\rb&=\alpha(0) \\
		&=0\\
		&=\sqrt{q}^{-1}(S_i-q)(S_i+1) \\
		&=(\sqrt{q}^{-1}S_i-\sqrt{q})(\sqrt{q}^{-1}S_i+\sqrt{q}^{-1})\\
		&=\beta^{n, q}(T_i-\sqrt{q})\beta^{n, q}(T_i+\sqrt{q}^{-1}).
	\end{align*}
	
	Since 
	\begin{align*}
		w_i(\omega_{i+1})&=\omega_i = \omega_{i+1} + (\omega_i-\omega_{i+1}) = \omega_i+\alpha_i,
	\end{align*}
	we have that $T_i\theta_{i+1}T_i=\theta_{\omega_i},$
	and thus
	\begin{align*}
		\beta^{n, q}(T_i\theta_{i+1}T_i)&=\alpha(\theta_i)\\
		&=q^{i-\frac{n+1}{2}}X_i \\
		&=q^{i-\frac{n+1}{2}}S_iX_{i+1}S_i \\
		&=q^{i-\frac{n+1}{2}}\beta(\sqrt{q}T_i) X_{i+1}\beta(\sqrt{q}T_i)\\
		&=\beta^{n, q}(T_i) q^{i+1-\frac{n+1}{2}}X_{i+1}\beta^{n, q}(T_i)\\
		&=\beta^{n, q}(T_i) \beta^{n, q}(\theta_{i+1})\beta^{n, q}(T_i).
	\end{align*}
	
	The remaining relations are straightforward to verify.

	For $\vec{n}=(n_1, \ldots, n_r)$ where $n_1+\cdots + n_r=n$, we will write $\mb{H}_{\vec{n}, q}$ for the subalgebra of $\mb{H}_{n, q}$ which is the image of the injection
	$$\mb{H}_{n_1, q}\ten_\bc\cdots \ten_\bc \mb{H}_{n_r, q}\hookrightarrow \mb{H}_{n, q},$$
	given by 
	$$1\ten \cdots 1\ten \underbrace{T_j}_{i^{\tx{th}}} \ten 1 \ten \cdots \ten 1 \mapsto \tilde{T}_{s_{n_1+\cdots + n_i+j}},$$
	and 
	$$1\ten \cdots 1\ten \underbrace{\theta_{\omega_j}}_{i^{\tx{th}}} \ten 1 \ten \cdots \ten 1 \mapsto \tilde{T}_{\omega_{n_1+\cdots + n_i+j}}$$
	
	Thus, for a Levi subgroup $M\cong \GL_{n_1}(F)\times\cdots \times \GL_{n_r}(F)$, we have a commuting diagram
	$$\begin{tikzcd}
		\mb{H}_{n, q} \arrow[r, "\beta^{n, q}"] &  \mc{H}(G, \mc{I}) \\
		\bigotimes_{i=1}^r\mb{H}_{n_i, q} \arrow[r, "\otimes \beta^{n_i, q}"] \arrow[u, hook]& \mc{H}(M, \mc{I}_M) \arrow[u] 
	\end{tikzcd}$$
	%
	inducing a commuting diagram of functors
	$$\begin{tikzcd}
		\Rep(G)_{[T, 1]} \arrow[r, "(-)^{\mc{I}}"] & \Mod(\mc{H}(G, \mc{I})) \arrow[r, "\beta^{n, q}_\ast"] & \Mod(\mb{H}_{n, q}) \\
		\Rep(M)_{[T\cap M, 1]} \arrow[r, "(-)^{\mc{I}_M}"'] \arrow[u, "I_{M\subset P}^G"] & \Mod (\mc{H}(M, \mc{I}_M))  \arrow[u, "\mc{H}(G{,} \ \mc{I})\ten_{\mc{H}(M{,} \ \mc{I}_M)}(-)"] \arrow[r, "\ten \beta^{n_i, q}"'] & \Mod(\bigotimes_{i=1}^r\mb{H}_{n_i, q}) \arrow[u, "\mb{H}_{n, q}\ten_{\mb{H}_{\vec{n}, q}}(-)"']
	\end{tikzcd} $$
	
	
	{\allowdisplaybreaks
		Given a multisegment $a=\{[a_i, b_i]\}_{i=1}^r$, with lengths $n_i=b_i-a_i+1$, define 
		\begin{align}
			\vec{n}_a^\vee&=(n_r, \ldots, n_1)\label{eqn: n_a}\\
			\vec{z}_a^\vee&=(q^{-a_r+1}, \ldots, q^{-a_1+1})\label{eqn: z_a}\\
			x_a^\vee&:=J(n_1)\oplus\cdots \oplus J(n_r),\label{eqn: x_a}\\
			D(n_i)&=\sum_{j=1}^{n_i}\sqrt{q}^{n_i+1-2j}E_{j,j},\label{eqn: D(n)}\\
			s_a&=\bigoplus_i z_i\sqrt{q}^{1-n_i}D(n_i).\label{eqn: s_a}
	\end{align}}
	
	{\allowdisplaybreaks
		Following \cite{Ari}, given some $s$ as in the above, we define the $\mb{H}_{\vec{n}^\vee, q}$-module $\bc_{\vec{n}_a^\vee, \vec{z}_a^\vee}$ to have underlying space $\bc$, where $1\ten \cdots \ten T_i\ten\cdots \ten 1$ acts by $q$, and $\theta_x$ acts by $x(s)$.  
		
		The following result classifies those $\mb{H}_{n, q}$-modules determined by the Zelevinksy substandards $S_Z(a)$.

		\begin{theorem}\label{thm: heckemod}
			For any multisegment,
			$$\beta_\ast^{n, q}(S_Z(a)^{\mc{I}})\cong \mb{H}_{n, q}\ten_{\mb{H}_{\vec{n}_a^\vee, q}}\bc_{\vec{n}_a^\vee, \vec{z}_a^\vee}.$$ 
		\end{theorem}
	}
	{\allowdisplaybreaks
		\begin{proof}
			Given a segment $\Delta=[a, b]$, the representation $Z(\Delta)\cong \nu^{(a+b)/2}$ is the unique irreducible subrepresentation of $I_P^G(\chi)$, for $\chi=\nu^a\boxtimes\cdots \boxtimes\nu^b$.   
			We will first determine the $\mc{H}(G, \mc{I})$-module structure of $Z(\Delta)^{\mc{I}}$.  
			
			Letting $s$ be simple reflection of the Weyl group of $\GL_{n}(F)$, we compute
			\begin{align*}
				\nu^{\frac{a+b}{2}}(\tilde{T}_s)\cdot z&=\int_G\nu^{\frac{a+b}{2}}(g)zdg \\
				&=\int_{\mc{I}s\mc{I}}\nu^{\frac{a+b}{2}}(k_1sk_2)dk_1dk_2 \\
				&=\nu^{\frac{a+b}{2}}(s)\int_{\mc{I}s\mc{I}}dk_1dk_2 \\
				&=\tx{Vol}(\mc{I}s\mc{I}) \\
				&=q.
			\end{align*}
			
			Let $\pi$ be the action of the representation $J(I(\chi))\ten\delta^{-1/2}$.
			For any cocharacter $\mu\in X_\ast(T)$, written as the difference $\mu=\mu_1-\mu_2$ of dominant cocharacters, the result of \cite{Ree}*{Proposition 3.1} tells us that we can compute the action of $\delta^{1/2}(\varpi^\mu)\tilde{T}_{\mu}\in \mc{H}(T, T_{\mc{I}})\cong \bc[X_\ast (T)]$, on $v\in I_P^G(\chi)^{\mc{I}}$ 
			as
			$$\delta^{1/2}(\varpi^\mu)\tilde{T}_{\mu}\cdot v = \pi(\mu(\varpi))\cdot v.$$
			
			Since $Z(\Delta)$ is a subrepresentation of $S_Z(\Delta)$, if we suppose $v\in Z(\Delta)^{\mc{I}}\subseteq I_P^G(\chi)^{\mc{I}}$, then since $J(Z(\Delta))\ten\delta^{1/2}\cong \chi\ten\delta^{-1/2}$ is a subrepresentation of $J\circ I_P^G(\chi)\ten\delta^{-1/2}$, 
			we compute the action of $\tilde{T}_{\ve_i}$ on $z\in Z(\Delta)^{\mc{I}}$ as
			\begin{align*}
				\delta^{1/2}(\varpi^{\ve_i})\tilde{T}_{\ve_1}\cdot z
				&=\delta^{1/2}(\varpi^{\ve_i})\chi(\varpi^{\ve_i}) \delta^{-1/2}(\ve(\varpi))\\
				&= q\w -\lb a +i-1\rb.
			\end{align*}
			Hence, the $\mb{H}_{n, q}$-module $ \beta_\ast^{n, q}(Z(\Delta)^{\mc{I}})$ is determined by
			\begin{align*}
				T_i\cdot z&=\beta^{n, q}(T_i)\cdot z=\sqrt{q}^{-1}\tilde{T}_{s_i}\cdot z =\sqrt{q}z,
			\end{align*}
			and,
			\begin{align*}
				\beta^{n, q}_\ast(\theta_{\omega_i})\cdot z
				&=\beta^{n, q}(\theta_{\omega_i})\cdot z\\
				&=q^{i-(n+1)/2} \tilde{T}_{\ve_i}\cdot z\\
				&=\delta^{1/2}(\varpi^{\ve_i})\tilde{T}_{\ve_i}\cdot z\\
				&=q^{-(a+i-1)}z.
			\end{align*}
			
			Now consider a multisegment $a=\{[a_i, b_i]\}_{i=1}^r$ with $\vec{n}^\vee, \vec{z}$ defined as above, let $M\cong \GL_{n_1}(F)\times \cdots \GL_{n_r}(F)$, 
			and let $w_0$ be the longest element of the set
			$$\{w\in W | w\cdot(B\cap M)\subseteq B, w^{-1}\cdot M\subseteq B\}.$$
			By \cite{Jan}*{Proposition 2.1.2}, the following diagram commutes:
			$$\begin{tikzcd}
				\Rep(G)_{(T, 1)} \arrow[r, "(-)^{\mc{I}}"] & \Mod (\mb{H}_{n, q}) \\
				\Rep(M)_{(T, 1)} \arrow[u, "I_P^G"] \arrow[r, "(-)^{\mc{I}_M}"'] &  \Mod (\mb{H}_{\vec{n}, q}) \arrow[u, "\mb{H}_{n, q}\ten_{\mb{H}_{\vec{n}_a^\vee, q}}(-)^{w_0}"']
			\end{tikzcd}$$
			Therefore,
			\begin{align*}
				\beta_{n, q}^\ast(S_Z(a)^{\mc{I}}&\cong  \beta_\ast^{n, q}\lb I_P^G(Z(\Delta_1)\boxtimes\cdots \boxtimes Z(\Delta_r))\rb^{\mc{I}}\\
				&\cong \mb{H}_{n, q}\ten_{\mb{H}_{\vec{n}_a^\vee, q}} \ten\beta_\ast^{n_i, q}\lb(Z(\Delta_1)\boxtimes\cdots \boxtimes Z(\Delta_r))^{w_0}\rb^{\mc{I}_M}\\
				&\cong \mb{H}_{n, q}\ten_{\mb{H}_{\vec{n}_a^\vee, q}}\beta_\ast^{n_r, q}(Z(\Delta_r))^{\mc{I}_{n_r}}\boxtimes \cdots \boxtimes \beta_\ast^{n_1, q}(Z(\Delta_1))^{\mc{I}}.
			\end{align*}
			By our computations above, for $w$ in the Weyl group
			$$W\lb \prod_{i=1}^r\GL_{n_i}(\bc)\rb\cong \prod_{i=1}^rS_{n_i},$$
			corresponding to the transposition $(i, i+1)$, in the module $\mb{H}_{n, q}$-module $\boxtimes_{i=1}^rF_{n_{r-i+1}}(Z(\Delta_{r-i+1}))$, $T_w$ acts by $q$.
			Taking $s$ as above,
			\begin{align*}
				s&=\bigoplus_{i=1}^r\sum_{j=1}^{n_i} z_i\sqrt{q}^{1-n_i}\sqrt{q}^{n_i+1-2j} E_{jj}=\bigoplus_{i=1}^r\sum_{j=1}^{n_i} q^{-a_i+1-j}E_{jj},
			\end{align*}
			hence, for $1 \leq k \leq n_i$,
			\begin{align*}
				\beta^{n_i, q}_\ast (e^{\omega_{n_r+n_{r-1}+\cdots + n_{i-1}+k}})\cdot z&=\beta^{n_i, q}_\ast (e^{\omega_k})\cdot z\\
				&=q^{-(a_i+k-1)}z,\\
				&=z_i\sqrt{q}^{1-n_i}\sqrt{q}^{n_i+1-2k},\\
				&=q^{-a_i+i-k},\\
				&=e^{\omega_{n_1+\cdots + n_{i-1}+k}}(s).
			\end{align*}
			thus we see that
			$$\beta_\ast^{n_r, q}(Z(\Delta_r))\boxtimes\cdots \boxtimes \beta_\ast^{n_1, q}(Z(\Delta_1))\cong \bc_{\vec{n}_a^\vee, \vec{z}_a^\vee},$$
			and therefore
			$$\beta_\ast^{n, q}(S_Z(a))\cong \mb{H}_{n, q}\ten_{\mb{H}_{\vec{n}_a^\vee, q}}\bc_{\vec{n}_a^\vee, \vec{z}_a^\vee}.$$
			
		\end{proof}
	}

	\section{The Geometry of Vogan Varieties}
	
	In this section we turn our attention to the geometric aspects of the $p$-adic Kazhdan-Lusztig hypothesis, and prove the main result.   
	We begin in Section 3.1 by recalling some notation and results from \cite{CG}, including their version \cite{CG}*{Theorem 8.6.23} of the $p$-adic Kazhdan-Lusztig hypothesis.  
	We also prove that the varieties considered in \cite{Zel} are exactly those considered by \cite{CG}, and we prove in Proposition \ref{prop: slnpklh} that the duality operator of \cite{CG}*{Proposition 8.6.25} acts trivially on the Grothendieck group of modules, which allows us to compare different ways of computing stalks of the cohomology of perverse sheaves as made precise in Proposition \ref{prop: slnpklh}. 
	
	As the theorems of \cite{CG} only apply to complex semi-simple simply connected Lie groups, we spend Section 3.2 relating the representation theory of $\PGL_n(F)$ to the representation theory of $\GL_n(F)$, and the geometry of $\SL_n(\bc)=\rwh{\PGL}_n(F)$ to the geometry of $\GL_n(\bc)$.  
	This allows us to derive Theorem \ref{thm: cgpklh} and Corollary \ref{cor: cgpklh}, which are essentially special cases of the $p$-adic Kazhdan-Lusztig hypothesis for $\GL_n(F)$, from Proposition \ref{prop: slnpklh} of the previous section. 
	
	In Section 3.3, we recall \cite{Vog}*{Conjecture 8.11}, and consequently the necessary background to state it, including a review of the basics of the local Langlands correspondence.  
	In particular, we purpose a minor sign change in the statement of \cite{Vog}*{Conjecture 8.11}, in Conjecture \ref{pKLH}.  
	
	{\allowdisplaybreaks
		The main result of Section 3.4 is Proposition \ref{prop: 1raygeo}, which resolves Conjecture \ref{pKLH} for representations with \emph{simple} inertial support (Definition \ref{def: simp}). 
		We also arrive at Corollary \ref{cor: unramseg}, which formally resolves (with Theorem \ref{thm: pklh}) the expectation that the multiplicities $m(a;b)$ depend purely on the combinatorics, and not the underlying fields or supercuspidal representations.  
		
		The main result of this paper, Theorem \ref{thm: pklh} is proved in Section 3.5, which resolves the $p$-adic Kazhdan-Lusztig hypothesis for $\GL_n(F)$. 
	}
	
	\subsection{The work of Chriss and Ginzburg}

	In this section, we recall some of the definitions and results of \cite{CG}, and relate the varieties used by \cite{CG} to the varieties used in the formulation of the $p$-adic Kazhdan-Lusztig hypothesis in \cite{Zel}.  
	The main result is Proposition \ref{prop: slnpklh}, which relates two different calculations of stalks of perverse sheaves appearing in \cite{CG}*{Theorem 8.6.23}.  
	Proving this result requires us to first establish some technical results about the modules defined in \cite{CG}.  
	In particular, we prove that the duality operation introduced in \cite{CG}*{Corollary 8.6.25} is actually the identity on the Grothendieck group of said modules. 
	
	Let $\hat{G}$ be a complex Lie group,
	let $\mc{N}_{\hat{G}}$ be the nilpotent cone of $\mf{g}:=\tx{Lie}(\hat{G})$, take $\mc{B}$ to be the variety of Borel subalgebras of $\mf{g}$, and for a semisimple element $a=(s, t)$ in $\hat{G}\times \bc^\ast$ define
	\begin{align*}
		\mc{N}_{\hat{G}}^a&:=\lset x\in \mc{N}_{\hat{G}}| sxs^{-1}=tx\rset,\\
		\mc{B}^s&:=\lset b\in \mc{B} | sbs^{-1}=b\rset\\
		\tilde{\mc{N}}_{\GL_n}^a&:=\lset (x, b)\in \mc{N}_{\GL_n}^a\times\mc{B}^s| x\in b\rset,\\
		\tilde{\mc{N}}_{\SL_n}^a&:=\lset (x, b)\in \mc{N}_{\SL_n}^a\times\mc{B}^s| x\in b\rset
	\end{align*}
	and let $\mu:\tilde{\mc{N}}_{\hat{G}}^a\to \mc{N}_{\hat{G}}^a$ be projection on the first factor. 
	
	Define $\mc{C}$ to be the perverse sheaf on $\tilde{\mc{N}}^a$ such that for each connected component $X$, $\mc{C}|_X\cong \mathbbm{1}_X[\dim_\bc X]$. 
	Then $E_{s, q}:=\Ext_{D^b(\mc{N}_{\SL_n}^{(s, q)})}^\bullet(\mu_\ast\mc{C}, \mu_\ast\mc{C})$ also has the structure of a $\bc$-algebra given by the Yoneda product.  
	
	We now suppose that $\hat{G}$ is semi-simple and simply connected.  
	For every point $x\in \mc{N}_{\hat{G}}^a$, let $\mc{B}_x^s$ be the fiber $\mu^{-1}\{x\}$.  
	The centralizer $Z_{\hat{G}}(s)$ acts on $\mc{N}_{\hat{G}}^a$ by the adjoint map, and by \cite{CG}*{Lemma 8.1.8, Proposition 8.6.15} there is a $E_{s, q}$-module structure on Borel-Moore homology $H_\bullet^{\tx{BM}}(\mc{B}_x^s)$ of $\mc{B}_x^s$, where if $y$ is in the  $Z_{\hat{G}}(s)$-orbit of $x$,  $H_\bullet^{\tx{BM}}(\mc{B}_x^s)\cong H_\bullet^{\tx{BM}}(\mc{B}_y^s)$
	as $E_{s, q}$-modules.  
	
	\begin{definition}
		For a $G$-variety $X$, a \emph{geometric parameter} $\phi=(C, \mc{L})$ is a pair where $C$ is a $G$-orbit, and $\mc{L}$ is a $G$-equivariant local system on $C$. 
		Given a geometric parameter $\phi=(C, \mc{L})$, we will write $C_\phi=C$ and $ \mc{L}_\phi=\mc{L}$.  
	\end{definition}
	
	By \cite{CG}*{Proposition 8.6.15}, 
	the Borel-Moore homology $H_\bullet^{\BM}(\mc{B}_x^s)$ can be given the structure of an $E_{s, q}$-module. 
	The double-centralizer $Z_{\hat{G}}(s, x)$ also acts on $H_\bullet^{\BM}(\mc{B}_x^s)$, which induces an action of the component group $A_x:=Z_{\hat{G}}(s, x)/Z_{\hat{G}}(s, x)^\circ$.  
	Since the category of $Z_{\hat{G}}(s, x)$-equivariant local systems on an orbit $C$ is equivalent to the category of finite-dimensional representations $\Rep(A_x)^{\mathrm{fd}}$ of $A_x$, each geometric parameter $\gamma=(C, \mc{L})$ determines an irreducible representation $\rho(\gamma)$ of $A_x$.  
	Thus, if $x\in C$, one can form the $\rho(\gamma)$-typic component $H_\bullet^{\BM}(\mc{B}_x^s)_\gamma$, which also carries the structure of an $E_{s, q}$-module by \cite{CG}{Proposition 8.6.16}. 
	
	Given a geometric parameter $\gamma=(C, \mc{L})$, we write $P_\gamma$ for the intersection cohomology complex $\mc{IC}(C, \mc{L})$ associated to $\gamma$. 
	Since these are constructible complexes, the restriction of $\mathcal{H}^n(\iota^! P_\gamma)$ or $\mathcal{H}^n(\iota^\ast P_\gamma)$ to any given orbit is again a local system. 
	Therefore, an entirely similar fashion to the above, for every geometric parameter $\xi$, there exist vector spaces
	$\mathcal{H}^n(\iota^! P_\gamma)_\xi$ (resp. $\mathcal{H}^n(\iota^\ast P_\gamma)_\xi$) whose dimension is the multiplicity of the local system $\mc{L}_\xi$ in $\mathcal{H}^n(\iota^! P_\gamma)|_{C_\xi}$ (resp. $\mathcal{H}^n(\iota^\ast P_\gamma)|_{C_\xi}$).
	
	By \cite{CG}*{theorem 8.6.23}, 	the multiplicity $[M: N]$ of $N$ in $M$ as a $E_s$-module,
	\begin{align}
		[H_\bullet^{\tx{BM}}(\mc{B}_x^s)_\xi:L_\gamma]&=\sum_{k\in \bz}\dim \mc{H}^k(i_x^!P_\gamma)_\xi=\sum_{k\in \bz}[ \mc{H}^k(i_x^!P_\gamma): \mc{L}_\xi]  \\
		[H^\bullet(\mc{B}_x^s)_\xi:L_\gamma]&=\sum_{k\in \bz}\dim \mc{H}^k(i_x^\ast P_\gamma)_\xi=\sum_{k\in \bz}[ \mc{H}^k(i_x^\ast P_\gamma):\mc{L}_\xi].\label{eqn: cgpklh}
	\end{align}
	
	For $s\in \SL_n(\bc)$, 
	$$\mc{N}_{\GL_n}^{(s, q)}=\mc{N}_{\SL_n}^{(s, q)},$$
	and for every $g\in Z_{\GL_n(\bc)}(s)$, and $x\in \mc{N}_{\GL_n}^{(s, q)}$, there exists $h\in Z_{\SL_n(\bc)}(s)$ such that $g\cdot x = h\cdot x$. 
	Therefore, the $Z_{\GL_n(\bc)}$-orbit and the $Z_{\SL_n(\bc)}$-orbits coincide. 
	By \cite{AchA}*{Lemma 6.1}\footnote{This result was certainly known much earlier, but author is unaware of where else the proof appears.}, for every $Z_{\GL_n(\bc)}(s)$-orbit $C$, the only $Z_{\GL_n(\bc)}(s)$-equivariant local system on $C$ is the constant sheaf. 
	Since $Z_{\SL_n(\bc)}(s)=Z_{\GL_n(\bc)}(s)\cap \SL_n(\bc)$, we have the inclusion map $Z_{\SL_n(\bc)}\to Z_{\GL_n(\bc)}(s)$. 
	Since the change of groups functor 
	$$\Per_{Z_{\GL_n(\bc)}(s)}\lb\mc{N}_{\SL_n}^{(s, q)}\rb \xrightarrow{\For} \Per_{Z_{\SL_n(\bc)}(s)}\lb\mc{N}_{\SL_n}^{(s, q)}\rb,$$
	commutes with restriction, 
	$$\For (\mc{H}^k(i_x^\ast P_\gamma))\cong \mc{H}^k(i_x^\ast \For(P_\gamma)).$$
	Therefore, we conclude that in $\SL_n(\bc)$, for any geometric parameter $\gamma=(C, \1_C)$, where $\1_C$ is the constant sheaf on $C$, the local system $\mc{H}^k(i_x^\ast P_\gamma)|_D$ only has the trivial local system in its composition factors. 
	
	We will now prove that the varieties appearing in \cite{Zel} are all (equivariantly) isomorphic to one of the form $\mc{N}_{\GL_n}^{(s, q)}$, both as a means to connect the phrasing of the $p$-adic Kazhdan-Lusztig hypothesis to the work of \cite{CG} and \cite{Vog}, but more importantly to extend the result \cite{Zel}*{Theorem 2.2} which relates the ordering on multisegments to the closure ordering on orbits. 
	First, we recall the varieties defined in \cite{Zel}.  
	
	Following \cite{Zel}*{Section 1.8}, for a function $\varphi:\bz\to \bn$, with finite support, we define the graded $\bc$-vector space $V_\varphi=\bigoplus_{n\in \bz}V_n$ where $\dim_\bc V_n=\varphi(n)$, and let $E_\varphi$ be the collection of operators $T:V\to V$ such that $T(V_n)\subseteq V_{n+1}$.  
	Writing $i$ (resp. $j$) for the minimum (resp. maximum) integer $n$ for which $\varphi(n)\neq 0$, 
	the group $A_\varphi:=\prod_{n=i}^j\GL(V_n)$ acts on $E_\varphi$ by 
	$$(g_i, g_{i+1}, \ldots, g_j)\cdot(x_i, x_{i+1} \ldots, x_{j+1})=(g_ix_ig_{i+1}^{-1}, \ldots, g_jx_jg_{j+1}^{-1}).$$
	
	\begin{lemma}\label{lem: zelvarcg}
		Let $a_0=\{[a_i]\}_{i=1}^n$ be a maximal multisegment with inertial support given by the trivial representation of $\GL_1(F)$ for a $p$-adic field $F$ with residue field order $q$.   
		Then, 
		\begin{enumerate}
			\item there exists an $s\in \GL_n(\bc)$ such that the multisegments $b\leq a$ are in bijection with the $Z_{\GL_n(\bc)}(s)$-orbits of $\mc{N}_{\GL_n(\bc)}^{(s, q)} $, and
			\item writing $C_b$ for the unique orbit determined by a multisegment $b\leq a$, we have that $b\leq c$ if and only if $C_c\subseteq \bar{C}_b$. 
		\end{enumerate}
	\end{lemma}
	\begin{proof}
		First we suppose that for each $i$, either $a_{i+1}=a_i$ or $a_{i+1}=a_i+1$.  
		Therefore, the imaginary parts $\Im(a_i)=\Im(a_j)$ are all equal.  
		For any segment $\Delta = [x, y]$, define $\Delta'= [\lfloor \Re(x) \rfloor, \lfloor \Re(y) \rfloor]$, and for any $b=\{\Delta_1, \ldots, \Delta_r\}$ such that $b\leq a$, define $b':=\{\Delta_1', \ldots, \Delta_r'\}$.  
		Observe that for $b, c\leq a$ we have that $b'\leq c'$ if and only if $b\leq c$. 
		In particular, we have a bijection between $\{b\leq a\}$ and $\{c\leq a'\}$. 
		
		Define a multiset $\varphi:\bz\to \bn$ such that $\varphi(\lfloor \Re(a_i)\rfloor)$ is equal to the number of occurrences of $[a_i]$ in $a$.  
		Choosing a basis for each $V_n$ in $E_\varphi$, the union is a basis of $V$, and thus the matrix of each $T\in E_\varphi$ in the basis is of the form
		$$Y= \begin{pmatrix}
			0 & 0 & \cdots  &  0 & 0 \\
			y_1 & 0 & \cdots & 0 & 0 \\
			0 & y_2 & \cdots & 0 & 0 \\
			\cdots & & \cdots & & \cdots \\
			0 & 0 & \cdots & y_t & 0 \\
			0 & 0 & \cdots & 0 & 0 
		\end{pmatrix}$$
		and every such matrix determines an element $T\in E_\varphi$.  
		
		Let $s=\tx{diag}(s_1, \ldots, s_n)$ be the $n$-by-$n$ diagonal matrix with eigenvalues $q^{i}$, with multiplicity $\varphi(i)$, arranging the diagonal of $s$ such that $\log_q(s_i)<\log_q(s_{i+1})$.
		Then, the elements of $\mc{N}_{\GL_n}^{(s, q)}$ are exactly those matrices of the form $Y$, and thus choosing a basis yields an isomorphism $f:E_\varphi\to \mc{N}_{\GL_n}^{(s, q)}$.
		In the same way, we have an isomorphism
		\begin{align*}
			A_\varphi          & \xrightarrow{\psi} Z_{\GL_n(\bc)}(s) \\ %
			(g_1, \ldots, g_t) & \mapsto \tx{diag}(g_1, \ldots, g_t)
		\end{align*} 
		such that $f(g\cdot X)=\psi(g)\cdot f(X)$.  
		In other words, there is an equivariant isomorphism $(f, \psi):(E_\varphi, A_\varphi)\to (\mc{N}_{\GL_n}^{(s, q)}, Z_{\GL_n(\bc)}(s))$. 
		
		For any multisegments $b, c\leq a$, let $D_{b'}$ be the $A_\varphi$-orbit of $E_\varphi$ associated to $b'$ as in \cite{Zel}, and $C_b:=f(D_{b'})$, which is also an $Z_{\GL_n(\bc)}(s)$-orbit.
		Given $b, c\leq a$, we have $b\leq c$ if and only if $b'\leq c'$, and by \cite{Zel}*{Theorem 2.2},
		$$b'\leq c' \iff D_{c'}\subseteq \bar{D}_{b'},$$
		and thus we can conclude that
		$$b\leq c \iff C_c\subseteq \bar{C}_b.$$
		
		Now consider an arbitrary maximal multisegment $a=\{[a_{ij}]\}_{i=1}^n$ with trivial inertial support.  
		We can assume they are arranged such that for each $i$ and $j$, $a_{i,j+1}=a_{i,j}, a_{i, j}+1$.     
		Defining
		$$
		s_i  :=\bigoplus_{j=1}^{t_i}q^{a_{ij}}, \ \ \ \ \ \ \ \ \ \  \ \ \
		s_a =\bigoplus_{i=1}^rs_i,$$
		and writing $m_i:=t_1+\cdots + t_i$,
		$$\mc{N}_{\SL_n}^{(s_a, q)}=\tx{span}_\bc\{E_{m_i+j, m_i+k} : 1\leq j, k \leq m_i, a_{ij}-a_{ik}=1\}.$$
		Thus, we have an isomorphism
		\begin{align*}
			\prod_{i=1}^r\mc{N}_{\SL_{t_i}(\bc)}^{(s_i, q)}&\xrightarrow{f} \mc{N}_{\SL_n}^{(s, q)} \\
			(x_1, \ldots, x_r)&\mapsto \bigoplus_{i=1}^rx_i
		\end{align*}
		which is equivariant with respect to the corresponding isomorphism
		\begin{align*}
			\prod_{i=1}^r Z_{\SL_{t_i}(\bc)}(s_i)&\to Z_{\SL_n(\bc)}(s)\\
			(g_1, \ldots, g_{r+1})&\mapsto \bigoplus_{i=1}^{r+1}g_i
		\end{align*}
		of groups. 
		For the multisegments $a_i=\{[a_{ij}]\}_{j=1}^{s_i}$, if $[a_{ij}]$ and $[a_{kl}]$ are linked, then $i=k$.  
		Hence, for $a:=a_1+\cdots + a_r$, by Lemma \ref{prop: prodmult} every $b\leq a$ is of the form $b_1+\cdots + b_r$ where each $b_i\leq a_i$.  
		Thus, we can define
		$$ C_{b_1+\cdots +b_r}:=f\lb \prod_{i=1}^r C_{b_i}\rb.$$
		
		For each $b_i\leq a_i$, where $a_i=\{[a_{jk}]\}$ such that $a_{j, k+1}=a_{j, k}, a_{j, k}+1$.  
		Thus by the first case for multisegments considered, for some $c_i\leq b_i$ we have $C_{c_i}\subseteq \bar{C}_{b_i}$ if and only if $b_i\leq c_i$.  
		Therefore, for each $b\leq c$, by Lemma \ref{lem: multirays} we can choose $b_i, c_i \leq a_i$ such that $b=b_1+\cdots + b_r, c=c_1+\cdots + c_r$ and each $b_i \leq c_i$.  
		Therefore, for every $b\leq c $ if and only if each $b_i \leq c_i$, if and only if each $C_{c_i}\subseteq \bar{C}_{b_i}$ if and only if 
		$$C_{c_1}\times \cdots\times C_{c_r}\subseteq \bar{C}_{b_1}\times\cdots \times \bar{C}_{b_r}=\overline{C_{b_1}\times \cdots \times C_{b_r}}=\bar{C}_{b_1+\cdots + b_r}.$$
	\end{proof}
	
	Now that we have Lemma \ref{lem: zelvarcg}, we can prove the following technical result which will be crucial in the next section for using \cite{CG}*{Theorem 8.6.23} in the case $\hat{G}=\SL_n(\bc)$, to derive the $p$-adic Kazhdan-Lusztig hypothesis for $\GL_n(F)$.

	\begin{proposition}\label{prop: slnpklh}
		For $s\in\SL_n(\bc)$, and any geometric parameters $\gamma, \xi$ for which $\mc{L}_\gamma, \mc{L}_\xi$ are the constant sheaf, 
		\begin{enumerate}
			\item $L_\gamma^\vee\cong L_\gamma$,
			\item $
			[H_\bullet^{\tx{BM}}(\mc{B}_x^s)_\xi:L_\gamma]
			=\sum_{k\in \bz}\dim \mc{H}^k(i_x^\ast P_\gamma)_\xi=\sum_{k\in \bz}\dim \mc{H}^k(P_\gamma)|_{C_\xi}.$
		\end{enumerate}
	\end{proposition}
	\begin{proof}
		\begin{enumerate}
			\item By Lemma \ref{lem: zelvarcg}, and the discussion that follows it, there is a unique open orbit $C$.  
			Writing $\1_C$ for the constant sheaf on an orbit $C$, set $\xi=(C, \1_C)$, and let $x\in C$.   
			The only $P_\gamma$ that is supported on the open orbit is when $\gamma=(C, \mc{L})$.  
			Moreover, $\mc{H}^n(P_\gamma)|_{C}\cong \mc{L}$ for $n=\dim C$, and 0 otherwise.  
			Therefore, we conclude that $
			H_\bullet^{\mathrm{BM}}(\mc{B}_x^s)_\xi\cong L_\xi$ and $H^\bullet(\mc{B}_x^s)_\xi\cong L_\xi$.

			By \cite{CG}*{Corollary 8.6.25}
			$$L_\xi^\vee\cong H_\bullet^{\tx{BM}}(\mc{B}_x^s)_\xi^\vee\cong H^\bullet(\mc{B}_x^s)_\xi\cong L_\xi.$$
			
			Let $D$ be an orbit such that there does not exist an orbit $E$ such that $D<E<C$, and choose $x\in D$.   
			The only $P_\gamma$ supported on $D$ are those for which the orbit of $\gamma$ is either $D$ or $C$.  
			In particular, $\mc{H}^n(i_x^\ast P_\gamma)\cong \mc{L}_\gamma$ for $n=\dim D$, and 0 otherwise.
			Therefore, in the Grothendieck group of $E_{s, q}$-modules
			$$[H^\bullet(\mc{B}_x^s)]=[L_{(D, \1_D)}] + p\cdot [L_{(C, \1_C)}].$$
			Likewise,
			$$[H_\bullet^{\mathrm{BM}}(\mc{B}_x^s)]=m\cdot [L_{(D, \1_D)}] + n\cdot [L_{(C, \1_C)}].$$
			Again by \cite{CG}*{Corollary 8.6.25}
			\begin{align*}
				m\cdot [L_{(D, \1_D)}^\vee] + n\cdot [L_{(C, \1_C)}]&= m\cdot [L_{(D, \1_D)}^\vee] + n\cdot [L_{(C, \1_C)}^\vee]\\
				&= [H_\bullet^{\mathrm{BM}}(\mc{B}_x^s)^\vee]\\
				&= [H^\bullet(\mc{B}_x^s)]\\
				&=[L_{(D, \1_D)}] + p\cdot [L_{(C, \1_C)}].
			\end{align*}
			
			Therefore $L_\xi=L_\xi^\vee, m=1, n=p$.  
			The result follows by continuing inductively in this manner. 
			
			\item Since the category $\Loc_{Z_{\SL_n(\bc)(s)} }(D)$ of $Z_{\SL_n(\bc)}(s)$-equivariant local systems on $D$ is semi-simple, and the trivial local system is 1-dimensional,
			$$[\mc{H}^k(i_x^\ast P_\gamma): \mc{L}_\xi]=\dim \mc{H}^k(P_\gamma)|_{C_\xi}.$$
		\end{enumerate}
	\end{proof}
	

	\subsection{Relation to the representation theory of $\PGL_n(F)$}
	
	Technically, as written, the theorems of \cite{CG} only apply to the case when the dual group $\hat{G}$ is semisimple and simply connected.  
	As such, we must perform an intermediary step, passing through the representation theory of $\PGL_n(F)$ in order to apply the results of \cite{CG}, which we carry out in this section.  
	This section concludes with Theorem \ref{thm: cgpklh} and Corollary \ref{cor: cgpklh} which prove an analogue of Proposition \ref{prop: slnpklh} for $\GL_n(F)$, and essentially confirms \cite{Zel}*{Hypothesis 1.9} for multisegments with trivial inertial support. 
	
	Consider the quotient map 
	$$p:\GL_n(F)\to \GL_n(F)/Z(\GL_n(F))\cong \PGL_n(F).$$
	Writing $\Rep(\GL_n(F))^Z$ for the full subcategory of $\Rep(\GL_n(F))$ with trivial central character, restriction of scalars $p_\ast$ induces an equivalence of categories
	$$ \Rep(\PGL_n(F))\xrightarrow{p_\ast}\Rep(\GL_n(F))^Z.$$
	Moreover, for the Iwahori-subgorup $\mc{I}$ of $\GL_n(F)$, $\mc{I}/Z$ is an Iwahori subgroup of $\PGL_n(F)$.
	For the inclusion $\mb{H}_{\SL_n(\bc)}\xhookrightarrow{\iota} \mb{H}_{\GL_n(\bc)}$, the restriction of scalars functor satisfies the commutative diagram

	$$\begin{tikzcd}
		\Rep(\PGL_n(F))_{(T, 1)} \arrow[r, "p_\ast"] \arrow[d, "\simeq"] & \Rep(\GL_n(F))_{(T, 1)}^Z \arrow[d, "\simeq"] \\
		\Mod(\mb{H}_{\SL_n(\bc)})  & \Mod(\mb{H}_{n, q}) \arrow[l, "\iota_\ast"]
	\end{tikzcd}$$
	
	Therefore, the irreducible representations of $\PGL_n(F)$ are in bijection with the irreducible representations of $\GL_n(F)$ having trivial central character.  
	
	Letting $M=M_{n_1, \ldots, n_r}\cap \SL_n(\bc)$, by the above commuting diagram, 
	$$\iota_\ast\lb \mb{H}_{n, q}\ten_{\mb{H}_{\vec{n}, q}}\bc_{\vec{n}, \vec{z}}\rb\cong \mb{H}_{\SL_n(\bc)}\ten_{\mb{H}_{M, q}}\bc_{\vec{n}, \vec{z}}.$$
	
	For $s=\bigoplus_{i=1}^nq^{a_i}$, $q^{-z}s\in \SL_n(\bc)$, and by the previous section
	$$\mc{N}_{\GL_n}^{(s, q)}=\mc{N}_{\SL_n}^{(q^{-z}s, q)}.$$
	Therefore the underlying space of $H_\bullet^{\BM}(\mc{B}_x^s)$ is the same, taking $x$ to be in either variety, and by the above commuting diagram, we find that the $\mb{H}_{\SL_n(\bc)}$-module structure on $H_\bullet^{\BM}(\mc{B}_x^s)$ is isomorphic to the image of the $\mb{H}_{n, q}$-module $H_\bullet^{\BM}(\mc{B}_x^s)$ under $\iota_\ast$.  
	Moreover, each of the simple $\mb{H}_{\SL_n(\bc)}$-modules $L_\gamma$, determine a simple $\mb{H}_{n, q}$-module.  
	
	Every segment $b$ there is a maximal segment $a=\{[a_1], \ldots, [a_n]\}$ for which $b\leq a$.  
	Therefore, $Z(b)$ is a subquotient of the indecomposible representation $S_Z(a)$, and thus they must have isomorphic central characters.  
	The central character of 
	$$S_Z(a)=I_B^G(\nu^{a_1}\boxtimes \cdots \boxtimes \nu^{a_n}),$$
	is 
	$$z I_n \mapsto \delta_B^{1/2}(\nu^{a_1}\boxtimes \cdots \boxtimes \nu^{a_n})(z, z, \ldots, z).$$
	Since 
	$$\delta_B^{1/2}=\nu^{(n-1)/2}\boxtimes \nu^{(n-3)/2}\boxtimes\cdots \boxtimes\nu^{(1-n)/2},$$
	is trivial on $(z, z, \ldots, z)$, we see that the central character is trivial if and only if $a_1+\cdots + a_n=0$.  
	In other words, $Z(b)$ determines an irreducible representation of $\PGL_n(F)$ if and only if the complex numbers defining $b$ sum to 0. 
	
	For any multisegment $b=\{\Delta_i\}_{i=1}^r$ where $\Delta_i=[b_i, c_i]$, we make the following definitions similar to Equations \ref{eqn: z_a}, \ref{eqn: x_a}, \ref{eqn: s_a}, 
	\begin{align*}
		\vec{n}_b^\vee&:=(n_r, \ldots, n_1), \\
		x_b^\vee&:=\bigoplus_{i=1}^rJ(n_{r-i+1}),\\
		\vec{z}_b^\vee&:=(q^{-b_r+1}, \ldots, q^{-b_1+1}),\\
		\bar{s}_b&=\bigoplus_{i=1}^r z_i\sqrt{q}^{1-n_{r-i+1}}D(n_{r-i+1}).
	\end{align*} 
	Then, by the above, $Z(b)$ corresponds to a representation of $\PGL_n(F)$ if and only if $$\sum_{i=1}^r\sum_{j=0}^{c_i}(b_i+j)=0,$$
	which holds if and only if
	$\bar{s}_b\in \SL_n(\bc)$.  
	In this case $x_b^\vee\in \mc{N}_{\GL_n}^{(\bar{s}_b, q)}$.

	Now, supposing $a=\{[a_{ij}]\}$ is a maximal multisegment satisfying $(a_{ij}-a_{kl})\in \bz$ if and only if $i=k$.  
	Thus we can write $a=a_1+\cdots + a_r$ such that $\Delta\in a_i, \Delta'\in a_j$ are linked then $i=k$, and we can decompose any $b\leq a$, by $b=b_1+\cdots + b_r$ where $b-i\leq a_i$.  
	For each $i$, choose $w_i$ to be a permutation matrix sending $s_{b_i}\mapsto s_{a_i}$.  
	Then, for $w:=\bigoplus_{i=1}^r w_i$, we have that $\Ad(w)\bar{s}_b=\bar{s}_a, \Ad(w)s_b=s_b$, and letting $t$ stand for transpose, a commuting diagram of equivariant isomorphisms
	$$\begin{tikzcd}
		\mc{N}^{\bar{s}_b} \arrow[d, "w_a"'] \arrow[r, "t"] & \mc{N}^{s_b} \arrow[d, "w_b"] \\
		\mc{N}^{\bar{s}_a} \arrow[r, "t"] & \mc{N}^{s_a}
	\end{tikzcd}$$
	In particular, 
	\begin{align*}
		\Ad(w)(x_b^\vee)^t&=\Ad(w)\lb \bigoplus_{i=1}^r \bigoplus_{j=1}^{s_i} J(n_{i, s_i-j+1})\rb^t\\
		&=\bigoplus_{i=1}^r \Ad(w_i)\bigoplus_{j=1}^{s_i} J(n_{i,s_i-j+1})^t\\
		&=\bigoplus_{i=1}^r \bigoplus_{j=1}^{s_i} J(n_{i,j})^t \\
		&=x_b.
	\end{align*}
	Therefore $\Ad(w)(C_b^\vee)^t = C_b$.  
	
	{\allowdisplaybreaks
		\begin{theorem}\label{thm: cgpklh}
			Let $a$ be a multisegment such that $\bar{s}_a\in \SL_n(\bc)$, $F_n$ be the functor from Section \ref{ssec: heckmod}, and let $C_a^\vee$ be the orbit of $x_a^\vee$ in $\mc{N}_{\GL_n}^{(\bar{s}_a, q)}$.    
			Writing $P_{a^\vee}:=P_{(C_a^\vee, \1_{C_a^\vee})}$, and $L_{a^\vee}:=L_{(C_a^\vee, \1_{C_a^\vee})}$, 
			\begin{enumerate}
				\item $F_n(S_Z(a))\cong  H_\bullet^{\mathrm{BM}}(\mc{B}_{x_a^\vee}^{s^\vee})_{(C_a^\vee, \1)}$,
				\item $F_n(Z(a))\cong L_{a^\vee}$, and thus
				\item $[S_Z(a):Z(b)]=\sum_{k\in \bz}\dim\mc{H}^k(P_{b^\vee})|_{C_a^\vee}.$
			\end{enumerate}
		\end{theorem}
	}
	{\allowdisplaybreaks
		\begin{proof}
			\begin{enumerate}
				\item By Theorem \ref{thm: heckemod} and \cite{Ari}*{Theorem 3.2}, 
				$$[F_n(S_Z(a))]=[\mb{H}_{n, q}\ten_{\mb{H}_{\vec{n}_a^\vee, q}}\bc_{\vec{n}_a^\vee, \vec{z}_a^\vee}]=[H_\bullet^{\mathrm{BM}}(\mc{B}_{x_a^\vee}^{\bar{s}_a})].$$
				
				\item If none of the segments are linked, then $a$ is minimal, $Z(a)\cong S_Z(a)$. 
				By Lemma \ref{lem: zelvarcg} 2) $C_a$ is the maximal/open orbit, and since equivariant isomorphisms preserve the ordering, so is $C_a^\vee$.  
				Thus in the Grothendieck group of $\mb{H}_{n, q}$-modules
				$$[F_n(Z(a))]=[ F_n(S_Z(a))]= [H_\bullet^\BM(\mc{B}_{x_a^\vee}^{\bar{s}_a})],$$
				but since $F_n(Z(a))$ is irreducible, it must be that $F_n(Z(a))\cong H_\bullet^\BM(\mc{B}_{x_a^\vee}^{\bar{s}_a})$.  
				By Proposition \ref{prop: slnpklh},
				\begin{align*}
					[H_\bullet^{\tx{BM}}(\mc{B}_{x_a^\vee}^{\bar{s}_a})_\xi:L_{b^\vee}]
					&=\sum_{k\in \bz}\dim \mc{H}^k(i_{x_a^\vee}^\ast P_{b^\vee})|_{C_a^\vee},
				\end{align*}
				but since $P_{b^\vee}$ is supported on $\bar{C}_b^\vee$, this is only non-zero for $a=b$, in which case it is equal to 1, thus $H_\bullet^{\BM}(\mc{B}_{x_a^\vee}^{\bar{s}_a})_\xi\cong L_{b^\vee}$.  
				
				If $a$ is obtained from $b$ by a simple operation, then, there are no orbits $D$ in $\mc{N}_{\GL_n}^{(s, q)}$ such that $C_b^\vee<D<C_a^\vee$.  
				Since $F_n$ is exact, it induces a map on the Grothendieck groups.  
				We have that
				$$[S_Z(b)]=[Z(b)] + m[Z(a)],$$
				for some integer $m$.  
				As in the proof of Proposition \ref{prop: slnpklh}, there is an integer $p$ so that in the Grothendieck group of $\mb{H}_{\SL_n(\bc), q}$-modules,
				$[H_\bullet^{\BM}(\mc{B}_{x_b^\vee}^{\bar{s}_b})]=[L_{b^\vee}]+p[L_{a^\vee}]$, which therefore holds for the corresponding $\mb{H}_{n, q}$-modules, and so
				\begin{align*}
					[F_n(Z(b))] + m[F_n(Z(a))]&=[F_n(S_Z(b))]\\
					[F_n(Z(b))] + m[L_{a^\vee}]&=[H_\bullet^\BM(\mc{B}_{x_b^\vee}^{\bar{s}_b})] \\
					&=[L_{b^\vee}]+p[L_{a^\vee}].
				\end{align*}
				Therefore $F_n(Z(b))\cong L_{b^\vee}$ and $m=p$.  
				The result follows inductively as in the proof of Proposition \ref{prop: slnpklh}.

				\item By the above,
				\begin{align*}
					[S_Z(a):Z(b)]&=[F_n(S_Z(a)): F_n(Z(b))] \\
					&=[H_\bullet^\BM(\mc{B}_{x_a^\vee}^{s_a}):L_{b^\vee}] \\
					&=\sum_{k\in \bz}\dim \mc{H}^k(P_{a^\vee})|_{C_b^\vee}.
				\end{align*}
				Taking $w$ to be as in the preceding discussion, we have an equivariant isomorphism $\Ad(w)(-)^t:\mc{N}_{\GL_n}^{(\bar{s}_a, q)}\to \mc{N}_{\GL_n}^{(s_a, q)}$ taking $C_b^\vee$ to $C_b$, and thus
				$$[S_Z(a):Z(b)]=\sum_{k\in \bz}\dim \mc{H}^k(P_{a^\vee})|_{C_b^\vee}=\sum_{n\in \bz}\dim\mc{H}^n(P_b)|_{C_a}.$$
			\end{enumerate}
		\end{proof}
	}
	
	\begin{corollary}\label{cor: cgpklh}
		For any multisegments $a,b$ of $\GL_n(F)$ with trivial inertial support, 
		$$[S_Q(a):Q(b)]=\sum_{n\in \bz}\dim \mc{H}^n(P_b)|_{C_a}.$$
	\end{corollary}
	\begin{proof}
		
		Every multisegment $b$ of an unramified representation satisfies $b\leq a$ for a multisegment of the form
		$$a=\{[a_1], \ldots, [a_n]\}.$$
		By the arguments above, defining $z:=(a_1+\cdots + a_n)/n$, $S_Z(a_{-z})$ has trivial central character, as does each subquotient which includes $Q(b')$ for each $b'\leq a_{-z}$.  
		In particular, each $b'=b_{-z}$ for some $b\leq a$.  
		By Lemma \ref{lem: twist_seg}, 
		$$\nu^{-z}\ten Q(b)\cong Q(b_{-z}),$$
		and since tensoring by $\nu^{-z}$ is an exact functor, we find that for all $b, c\leq a$, $m(b;c)=m(b_{-z}; c_{-z})$, which is to say that
		$$[S_Q(a):Q(b)]=[S_Z(a):Z(b)]=\sum_{n\in \bz}\dim\mc{H}^n(P_{b_{-t}})|_{C_{a_{-t}}}=\sum_{n\in \bz}\dim\mc{H}^n(P_{b})|_{C_{a}},$$
		concluding the result.   
	\end{proof}
	
	The above Corollary, together with the rest of the work in this section, demonstrates how \cite{Zel}*{Hypothesis 1.9} for multisegments with trivial inertial support, is a consequence of \cite{CG}*{Theorem 8.6.23}.

	Now, consider two diagonal matrices
	\begin{align*}
		s&=\tx{diag}((p^e)^{a_1}, \ldots, (p^e)^{a_r}),\\
		s'&=\tx{diag}((p^f)^{a_1}, \ldots, (p^f)^{a_r}).
	\end{align*}
	Observe that if $x$ is in the $p^e$-eigenspace of conjugation by $s$, then $x_{ij}$ can only be non-zero when $a_i-a_j=1$. 
	Since this is independent of $e$, we conclude that 
	\begin{align*}
		\mc{N}_{\GL_n}^{(s, p^e)}&=\mc{N}_{\GL_n}^{(s', p^f)} \\
		Z_{\GL_n(\bc)}(s)&=Z_{\GL_n(\bc)}(s').
	\end{align*}
	Hence, it follows that for any unramified multisegments,
	$$m(a;b)_{p^e}=\sum_{n\in \bz}\dim \mc{H}^n(P_b)|_{C_a}=m(a;b)_{p^f}.$$
	

	\subsection{The Langlands correspondence and Vogan varieties}
	
	Our goal for this section is to recall Vogan's version of the $p$-adic Kazhdan-Lusztig hypothesis in Conjecture \ref{pKLH}.  
	In order to articulate this, we must first review the local Langlands correspondence for $\GL_n(F)$. 
	
	The local Langlands correspondence is fundamentally about the \emph{reciprocity} map $\rec$ which maps (isomorphism classes of) irreducible representations to (equivalence classes of) \emph{Langlands parameters}.  
	The general definition will not concern us here, and in fact, we will find is easier to focus on \emph{Weil-Deligne} representations, though some comments will be made on the equivalence of these perspectives.

	\begin{definition}
		A \emph{Weil-Deligne} representation $(\sigma, V, N)$ is a representation $\sigma:W_F\to \GL(V)$ together with a nilpotent linear endomorphism $N:V\to V$ such that $\Ad(\sigma(w))N=\norm{w}N$. 
	\end{definition}

	\begin{definition}
		Let $\eta_n:W_F\to \GL_n(\bc)$ by defined by 
		$$\eta_n:=\bigoplus_{i=0}^{n-1}\omega^i,$$
		and for the standard basis $e_0, \ldots, e_{n-1}$ of $\bc^n$, define the operator $N_n$ such that
		by 
		\begin{align*}
			N_n\cdot e_i&=e_{i+1}, \ \ \ \ 0\leq i \leq n-2, \\
			N_n\cdot e_{n-1}&=0. 
		\end{align*}
		
		We write $\tx{sp}(n)$ for the Weil-Deligne representation $(\eta_n, N_n)$.  
	\end{definition}
	
	For a multisegment $a=\{[\rho_i(b_i), \rho_i(c_i)]\}_{i=1}^r$, let $\sigma_i=\rec(\rho_i)|_{W_F}$, and $n_{i}=b_i-c_i+1$.  
	Then, the reciprocity map 
	$$\tx{rec} \lb Q(\Delta_1, \ldots, \Delta_r)\rb=\bigoplus_{i=1}^r (\omega^{a_i}\eta_{n_i}\ten\sigma_i,N_i\ten I_{n_i} ),$$
	describes a bijection between the smooth irreducible representations of $\GL_n(F)$ and Weil-Deligne representations. 
	We call the first factor $\lambda_a:=\bigoplus_{i=1}^r\omega^{a_i}\eta_{n_i}\ten\sigma_i$ the \emph{infinitesimal parameter} of $Q(a)$.  
	For the choice of basis appearing the the definition of $\tx{sp}(n)$, the second factor of $\rec(Q(a))$ is represented by the matrix $X_a:=\bigoplus_{i=1}^r J(n_i)^t$.  
	
	The perspective taken in \cite{Vog} is that a Langlands parameter is a particular kind of homomorphism
	$$\phi:\bc\rtimes W_F\to \GL_n(\bc),$$
	where the semi-direct product is given by the action $w\cdot z= \omega(w) z$.  
	The corresponding infinitesimal parameter is defined to be $\lambda_\phi:=\phi|_{W_F}$.  
	Each Weil-Deligne representation $(\sigma, N)$ determines such a homomorphism by defining
	$$\phi(w, z):=\sigma(w)\exp(zN),$$
	and every Langlands parameter arises in this way.  
	That is, there is a bijection between Weil-Deligne representations and Langlands parameters in this sense.  
	Let $\phi, \psi$ be a pair of Weil-Deligne representations, Langlands parameters, of infinitesimal parameters.  
	We say they are equivalent if there exists $g\in\GL_n(\bc)$ such that $\phi=\Ad(g)\circ \psi$.  
	We will say a little more about this when we discuss Vogan varieties.

	We can see from the above that for some maximal multisegment $a$, the irreducible representations $Q(b)$ with infinitesimal parameter equivalent to that of $Q(a)$ are exactly those for which $b\leq a$.  
	
	Let $G$ be a reductive $p$-adic group, and $\lambda:W_F\to {}^LG$ be an infinitesimal parameter.  
	Define the reductive group
	$$\hat{G}^{\lambda(I)}=\{g\in \hat{G} | \forall w\in I,  \Ad(\lambda(w))g=g\}.$$
	For the Lie algebra $\hat{\mf{g}}^{\lambda(I_F)}$, we define
	\begin{align}
		V_\lambda&:=\{x\in \hat{\mf{g}} |  \Ad(\lambda(\mf{f}))x=qx\},\label{eqn: Vl}\\
		H_\lambda&:=\{x\in \hat{G} | \forall w\in W_F,  \Ad(\lambda(w))x=qx\}\label{eqn: Hl}
	\end{align}
	where $q$ is the order of the residue field of $F$. 
	Note that $V_\lambda$ is exactly the $q$-eigenspace of $\Ad(\lambda)$ in $\hat{\mf{g}}^{\lambda(I_F)}$, as described in \cite{Vog}*{Equation 4.e}.  
	We have an action of $h\in H_\lambda$ on $x\in V_\lambda$, given by $h\cdot x:= \Ad(h)x$.  
	By \cite{Vog}*{Corollary 4.6}, and the discussion proceeding it, the $H_\lambda$-orbits are in bijection with Langlands parameters/Weil-Deligne representations by sending $x\in V_\lambda$ to $(\lambda, x)$.  
	
	Given an orbit $\mc{O}$ and an irreducible equivariant perverse sheaf $P$ on $V_\lambda$, we define the \emph{equivariant Euler characteristic} $\chi_\mc{O}$,
	$$\chi_\mc{O}(P)=\sum_{n\in \bz}(-1)^n\dim \mc{H}^n(P)|_\mc{O}.$$
	For $\mc{L}\in \Loc_H(\mc{O})$, define 
	$\chi_{(\mc{O}, \mc{L})}(P)$ to be the multiplicity of $\mc{L}$ in $\chi_\mc{O}(P)$. 
	
	Recall that by a \emph{geometric parameter}, we mean a pair $(\mc{O}, \mc{L})$ where $C\subseteq V_\lambda$ is an orbit and $\mc{L}$ is an irreducible $H_\lambda$-equivariant local system on $C$, and we write $P_\gamma$ for the associated intersection cohomology complex. 
	For a geometric parameter $\gamma$, we write $\gamma=(\mc{O}_\gamma, \mc{L}_\gamma)$.  
	Let $\chi_\psi(P_\gamma)$ be the multiplicity of $\mc{L}_\xi$ in 
	$$\sum_\bz (-1)^n [\mc{H}^n(P_\gamma)|_{\mc{O}_\xi}],$$
	as a sum in the Grothendieck group of $\Loc_H(\mc{O}_\xi)$.  
	
	\begin{definition}[The geometric character matrix, \cite{Vog}*{Definition 8.7}]
		We define the \emph{geometric character matrix} 
		$$c_g[\gamma, \xi]=(-1)^{d_\gamma}\chi_\gamma(P_\xi),$$
		where $d_\gamma=\dim \mc{O}_\gamma$.
		
		We will also write $\hat{c}_g[\gamma, \xi]:=(-1)^{d_\gamma}c_g[\gamma, \xi]=\chi_\gamma(P_\xi)$ for convenience.  
	\end{definition}
	
	By the local Langlands correspondence, there is a bijection between representations $\pi$ of infinitesimal parameter $\lambda$, and geometric parameters $\gamma$ for $V_\lambda$.  
	Thus, we write $\pi_\gamma$ for the unique irreducible representation corresponding to $\gamma=(\mc{O}, \mc{L})$, and write $S_\gamma$ for its standard representation.  
	Then, for any pair of geometric parameter $\gamma$ and $\xi$, define $m_r[\gamma, \xi]$ to be the multiplicity of $\pi_\gamma$ in $S_\xi$.  
	
	The original  \cite{Vog}*{Conjecture 8.11},  of the $p$-adic Kazhdan-Lusztig hypothesis is that
	given a connected, reductive, quasi-split algebraic group $G$ over a $p$-adic local field $F$, and an infinitesimal parameter $\lambda:W_F\to{}^LG$,
	$$m_r[\gamma, \xi]=(-1)^{d_\xi}c_g[\xi, \gamma]=\hat{c}_g[\xi, \gamma]=\chi_\xi(P_\gamma).$$
	
	In other words, the conjecture predicts that ${}^tm_r=\hat{c}_g$.  
	It was mentioned in \cite{Solp} that \cite{Vog}*{Conjecture 8.11} contains some signs which are useful for real reductive groups, but are better left out in the $p$-adic case, as the following example highlights. 
	
	\begin{example}
		For the infinitesimal parameter $\lambda=\omega^{1/2}\oplus\omega^{-1/2}$, there are two orbits $\mc{O}_0, \mc{O}_1$, and one computes that
		\begin{center}
			\begin{tabular}{c|c c}
				$c_g$ & $P_{(\mc{O}_0, \1_{\mc{O}_0})}$ & $P_{(\mc{O}_1, \1_{\mc{O}_0})}$ \\
				\hline 
				$\mathbbm{1}_{(\mc{O}_0, \mathbbm{1}_{\mc{O}_1})}$ &  1 & -1 \\
				$\mathbbm{1}_{(\mc{O}_1, \mathbbm{1}_{\mc{O}_1})}$ & 0 & 1
			\end{tabular}
		\end{center}
		and thus
		\begin{center}
			\begin{tabular}{c|c c}
				$\hat{c}_g$ & $P_{(\mc{O}_0, \mathbbm{1}_{\mc{O}_0})}$ & $P_{(\mc{O}_1, \mathbbm{1}_{\mc{O}_1})}$ \\
				\hline 
				$\mathbbm{1}_{(\mc{O}_0, \mathbbm{1}_{\mc{O}_0})}$ &  1 & -1 \\
				$\mathbbm{1}_{(\mc{O}_1, \mathbbm{1}_{\mc{O}_1})}$ & 0 & -1
			\end{tabular}
		\end{center}
	\end{example}
	The above matrix has negative entries, and therefore cannot possibly be the multiplicities for irreducible representations in standard representations.  
	Thus we put forward the slight rephrasing.  
	
	\begin{conjecture}[Vogan's (augmented) Kazhdan-Lusztig hypothesis]\label{pKLH}
		Given a connected reductive algebraic group $G$ over a locally compact non-Archimedean local field $F$, and an infinitesimal parameter $\lambda:W_F\to{}^LG$, for all $\gamma, \xi$,
			$$m_r[\gamma, \xi]=\sum_{n\in \bz}\dim \Hom_{\Loc_H(C_\xi)}(\mc{H}^n(P_\gamma)|_{C_\xi}, \mc{L}_\xi).$$
	\end{conjecture}
	
	\textbf{Note:}  Some other sources such as \cite{Cun} solve the sign issue in a slightly different way.  
	Of course, if both versions of the $p$-adic Kazhdan-Lusztig hypothesis are true, they should be equivalent.  
	One should be able to prove this is the case using certain ``purity" results about perverse sheaves which exist in the literature.  
	However, some work needs to be done in order to make this precise.


	\subsection{The case of simple support}
	
	In order to establish the general case of the $p$-adic Kazhdan-Lusztig hypothesis for $\GL_n(F)$, we will first prove the result for $V_\lambda$ where $\lambda$ is the infinitesimal parameter of a representation $Q(a)$ where $a$ has simple inertial support, which is accomplished in this section with Proposition \ref{prop: 1raygeo}.  
	From this, Corollary \ref{cor: unramseg} follows, which confirms that the multiplicities $m(b;a)$ really only depend on the "combinatorics" of multisegments, and not the underlying representations. 
	
	Suppose $a=\{\Delta_1, \ldots, \Delta_r\}$ is a multisegment of $\GL_{nd}(F)$ with simple inertial support, say with representative $\rho\in \Rep(\GL_d(F))$, with corresponding irreducible representation $\sigma:=\rec(\rho)|_{W_F}$. 
	Choose $a_1, \ldots, a_r, b_1, \ldots, b_r\in \bc$ such that $\Delta_i=[\rho(a_i), \rho(b_i)]$, and thus writing $n_i:=b_i-a_i+1$,
	$$\rec (Q(a))=\bigoplus_{i=1}^r(\omega^{a_i}\sigma_{n_i}\ten \sigma, J(n_i)^t\ten I_s),$$
	where we fix here and for the rest of the article the convention that for linear operators $A, B$, we will realize the matrix of $A\ten B$ explicitly in block-form with the Kroenecker product $[a_{ij}B]_{ij}$.  
	The geometric parameter $\gamma_a$ associated to $\rec (Q(a))$ is $(\mc{O}_a, \1)$ where $\mc{O}_a$ is the $H_{\lambda_a}$-orbit of 
	$$X_a:=\bigoplus_{i=1}^r J(n_i)^t\ten I_s,$$
	and $\1$ is the constant sheaf on $\mc{O}_a$. 
	The corresponding infinitesimal parameter is given by
	$$\lambda_a=\bigoplus_{i=1}^r\omega^{a_i}\eta_{n_i}\sigma=\bigoplus_{i=1}^r\bigoplus_{j=0}^{n_i-1}\omega^{a_i+j}\sigma.$$
	Define $\Delta_i^\circ:=[a_i, b_i]$, and $a^\circ:=\{\Delta_1^\circ, \ldots, \Delta_r^\circ\}$.  
	Let 
	$$X\in \hat{\mf{g}}^{\lambda_a(I_F)}=\{x\in \hat{\mf{g}} | \forall w\in I_F, \lambda_a(w)x=x\lambda_a(w)\},$$
	and writing $X$ in $d$-by-$d$ blocks, we find that for all $i, j$, and all $w\in I_F$
	\begin{align}\label{eqn: IF}
		\omega^{c_i}\sigma(w) X_{ij}=X_{ij}\omega^{c_j}\sigma(w)=X_{ij}\omega^{c_j+1}\omega(w)
	\end{align}
	where the last equality follows because $\omega(w)=1$ for all $w\in I_F$.  
	
	Suppose further that $X$ is in the $q$-eigenspace of $\Ad(\lambda(\mf{f}))$, we find that for all $i,j$, since $\omega(\mf{f})=q$,
	$$\omega^{c_i}\sigma(\mf{f}) X_{ij}=qX_{ij}\omega^{c_j}\sigma(\mf{f})=X_{ij}\omega^{c_j+1}\sigma.$$
	
	As $W_F$ is generated by $\mf{f}$ and $I_F$, by the above and Equation \ref{eqn: IF}, for all $w\in W_F$
	$$\omega^{c_i}\sigma(w) X_{ij}=X_{ij}\omega^{c_j+1}\sigma(w).$$
	In other words, $X_{ij}$ determines an intertwining operator between $\omega^{c_i}\sigma$ and $\omega^{c_j+1}\sigma$, and since these are irreducible representations, $X_{ij}$ can only be non-zero if $\Re(c_i)=\Re(c_j)+1$. 
	Moreover, since $\sigma$ is of Galois type, it must be of the form $\omega^z\sigma_0$ for some $z\in \bc$ and $\sigma_0:W_F \to \mathrm{Gal}(E/F)\to \GL_{d}(\bc)$.  
	In other words, Schur's lemma applies to $\omega^{c_i}\sigma, \omega^{c_j+1}\sigma$, and thus we can conclude that whenever $\Re(c_i)=\Re(c_j)+1$ the matrix $X_{ij}=x_{ij} I_s$ for some scalar $x_{ij}\in \bc$, and every such choice of $x_{ij}\in \bc$ determines an element $X\in V_\lambda$.  
	
	As seen in previous sections, we find that $V_{\lambda_\circ}$ consists of those $n$-by-$n$ matrices $[x_{ij}]$ such that $\Re(c_i)\neq\Re(c_j)+1\implies x_{ij}=0$.  
	Therefore, we obtain an isomorphism
	\begin{align*}
		V_{\lambda_a}&\xrightarrow{f_a} \mc{N}_{\GL_n}^{(s_{a^\circ}, q)}\\
		[x_{ij}I_d]&\mapsto [x_{ij}]
	\end{align*}
	By an entirely similar argument,
	$$H_\lambda=\{g\in \GL_{nd}(\bc) | g = [g_{ij}I_d], \Re(c_i)\neq \Re(c_j)\implies g_{ij}=0\},$$
	and thus we have an isomorphism
	\begin{align*}
		H_{\lambda_a}&\xrightarrow{\varphi_a}Z_{\GL_n(\bc)}(s_{a^\circ}) \\
		[g_{ij}I_d]&\mapsto [g_{ij}]
	\end{align*}
	for which the pair $(f_a, \varphi_a)$ determines an equivariant isomorphism $(V_{\lambda_a}, H_{\lambda_a})\to (\mc{N}_{\GL_n}^{(s_{a^\circ}, q)}, Z_{\GL_n}(s_{a^\circ}))$.
	In particular, 
	$$f_a\lb \bigoplus_{i=1}^r J(n_i)^t\ten I_s\rb =\bigoplus_{i=1}^r J(n_i)^t = x_a.$$
	
	Now, supposing $a$ is maximal and letting $w_b$ be a permutation matrix taking $s_{b^\circ}$ to $s_{a^\circ}$, define $\bar{w}_b:=w_b\ten I_s$. 
	Then, we have a commuting diagram of equivariant isomorphisms
	$$\begin{tikzcd}
		V_{\lambda_b} \arrow[r, "f_b"] \arrow[d, "\bar{w}_b"'] & \mc{N}_{\GL_n}^{(s_{b^\circ}, q)} \arrow[d, "w_b"] \\
		V_{\lambda_a} \arrow[r, "f_a"] & \mc{N}_{\GL_n}^{(s_{a^\circ}, q)}
	\end{tikzcd}$$
	
	In particular, for $X_b\in V_{\lambda_b}$, 
	$$f_a(\bar{w}_b\cdot X_b)=w_b\cdot f_b(X_b) = w_b\cdot x_b,$$
	where $x_b$ is defined as in Equation \ref{eqn: x_a}, and therefore, the orbit of $w_b\cdot f_b(X_b)$ is $C_{b^\circ}$.  
	
	Thus, for any geometric parameters $\gamma, \xi$ of $V_\lambda$, there exists some $b, c\leq a$ such that $\bar{w}_b\cdot X_b\in \mc{O}_\gamma$, and $\bar{w}_c\cdot X_c\in \mc{O}_\xi$.    
	Hence, 
	$$\dim \mc{H}^k(P(\mc{O}_\gamma))|_{\mc{O}_\xi}=\dim \mc{H}^k(P(C_{b^\circ}))|_{C_{c^\circ}}.$$
	
	Thus, we arrive at the following result.  
	
	\begin{proposition}\label{prop: 1raygeo}
		Let $\lambda:W_F\to \GL_n(\bc)$ be an infinitesimal parameter such that for a (equivalently all) irreducible representations $Q(a)$ with infinitesimal parameter $\lambda$, $a$ has trivial inertial support.  
		Then, for any geometric parameters $\gamma$ and $\xi$ of $V_\lambda$,
		$$[S_\gamma: \pi_\xi]=\sum_{n\in \bz}\dim \mc{H}^n(P_\xi)|_{\mc{O}_\gamma}.$$
	\end{proposition}
	\begin{proof}
		Let $a$ be the maximal multisegment such that the infinitesimal parameter $\lambda_a$ of $Q(a)$ is equivalent to $\lambda$.  
		By definition, there exists some $g\in \GL_n(\bc)$ such that $\Ad(g)\lambda =\lambda_a$.  
		Therefore, we have an isomorphism $\Ad(g):V_\lambda\to V_{\lambda_a}$.  
		
		Given any $\gamma, \xi$ be geometric parameters, and choose multisegments such that $\rec (Q(a))$ is equivalent to $\gamma$, and $\rec(Q(b))$ is equivalent to $\xi$.  
		
		By Corollary \ref{cor: unramseg}, there exists a $q'$ such that
		$$[S_Q(a): Q(b)]= m(a^\circ; b^\circ)_{q'},$$
		and by the discussion proceeding Corollary \ref{cor: unramseg}, we know that $m(a^\circ; b^\circ)_{q'}=m(a^\circ; b^\circ)_q$.  
		Therefore, 
		
		\begin{align*}
			[S_\gamma: \pi_\xi]&= [S_Q(a): Q(b)]\\
			&=m(a^\circ; b^\circ)_{q}, \\
			& =\sum_{n\in \bz}\dim\ \mc{H}^n(P_{b^\circ})|_{C_a^\circ}\\
			& =\sum_{n\in \bz}\dim \mc{H}^n(P_\xi)|_{\mc{O}_\gamma},
		\end{align*}
		where the last equality follows from preceding discussion. 
	\end{proof}

	\begin{corollary}\label{cor: unramseg}
		For any multisegments $a$ and $b$ with simple inertial support, letting $a^\circ, b^\circ$ be as in Theorem \ref{thm: simpunramseg},
		$$[S_Z(a): Z(b)]=m(a^\circ ; b^\circ)_q.$$
	\end{corollary} 
	\begin{proof}
		By Theorem \ref{thm: simpunramseg}, there exists a $q'$ such that 
		$$[S_Z(a):Z(b)]=m(a;b)_q=m(a^\circ; b^\circ)_{q'},$$
		which, by the discussion above, is equal to $m(a^\circ; b^\circ)_q$.  
	\end{proof}

	

	\subsection{The general case}
	
	In this section, we prove Theorem \ref{thm: pklh}, being the main result of this paper. 
	In order to do this, we first demonstrate that arbitrary Vogan varieties (for $\GL_n(F)$) can be decomposed in such a way that the general result follows form the previous section. 
	
	For an arbitrary multisegment $a$, by Lemma \ref{lem: multirays} and Proposition \ref{prop: prodmult} we can write $a=a_1+\cdots +a_r$ such that $\Delta\in a_i, \Delta'\in a_j$ have the same inertial support if and only if $i=j$.   
	Fixing representatives $\rho_1, \ldots, \rho_r$ for the inertial supports, we can write
	$$a_i=\{\Delta_{i1}, \ldots, \Delta_{is_i}\},$$
	where $\Delta_{ij}=[\rho_i(b_{ij}), \rho_i(c_{ij})]$, for some $b_{ij}, c_{ij}\in \bc$. 
	Let $\sigma_i=\rec(\rho_i)|_{W_F}$, and $n_{ij}=c_{ij}-b_{ij}+1$.  
	Then, 
	$$\sigma:=\rec(Q(a))=\bigoplus_{i=1}^r\bigoplus_{j=1}^{s_i}\omega^{b_{ij}}\eta_{n_{ij}}\sigma_i.$$

	In order to determine $V_\lambda$, we must first compute
	$$\hat{\mf{g}}^{\sigma(I_F)}=\{x\in \hat{\mf{g}} | \forall w\in I_F, \sigma(w)x=x\sigma(w)\}.$$
	If we suppose that $X\in \hat{\mf{g}}^{\sigma(I_F)}$, then writing $X$ in block form we must have that for each block $X_{ik}$, and each $w\in I_F$,
	$$\bigoplus_{j=1}^{s_i}\omega^{b_{ij}}\eta_{n_{ij}}\sigma_i(w)X_{ik}=X_{ik}\bigoplus_{j=1}^{s_i}\omega^{b_{kj}}\eta_{n_{kj}}\sigma_k(w),$$
	and since $\omega$ is trivial on $I_F$, in fact
	\begin{align}\label{eqn: resI}
		\bigoplus_{j=1}^{s_i}\omega^{b_{ij}}\eta_{n_{ij}}\sigma_i(w)X_{ik}=\omega(w)X_{ik}\bigoplus_{j=1}^{s_i}\omega^{b_{kj}}\eta_{n_{kj}}\sigma_k(w).
	\end{align}
	Now suppose further that $X$ is in the $q$-eigenspace of $\Ad(\sigma(w))$.  Then
	$$\bigoplus_{j=1}^{s_i}\omega^{b_{ij}}\eta_i\sigma_{n_ij}(\mf{f})X_{ik}=qX_{ik}\bigoplus_{j=1}^{s_i}\omega^{b_{kj}}\eta_k\sigma_{n_kj}(\mf{f})=\omega(\mf{f})X_{ik}\bigoplus_{j=1}^{s_i}\omega^{b_{kj}}\eta_k\sigma_{n_kj}(\mf{f}).$$
	This shows that $X_{ik}$ intertwines the representations 
	$$\bigoplus_{j=1}^{s_i}\omega^{b_{ij}}\eta_i\sigma_{n_ij}, \ \ \ \ \ \tx{and} \ \ \ \ \ \bigoplus_{j=1}^{s_i}\omega^{b_{kj}+1}\eta_k\sigma_{n_kj},$$
	at $\mf{f}$, and Equation \ref{eqn: resI} shows that it intertwines these representations when restricted to $I_F$.  
	Since $W_F$ is generated by $I_F$ and $\mf{f}$, it follows that $X_{ik}$ intertwines the representations.  
	However, if $i\neq k$ then by definition there is no $s\in \bc$ such that $\omega^s\eta_i\cong \eta_k$.  
	Thus none of the irreducible composition factors of 
	$$\bigoplus_{j=1}^{s_i}\omega^{b_{ij}}\eta_i\sigma_{n_ij}\cong \bigoplus_{i=1}^{s_i}\bigoplus_{j=0}^{n_{ij}-1}\omega^{b_{ij}+j}\eta_i,$$
	are isomorphic to any of the composition factors of 
	$$\bigoplus_{j=1}^{s_i}\omega^{b_{kj}+1}\eta_k\sigma_{n_kj},$$
	and therefore, there can not be any intertwining operators between them.   
	In other words $X_{ik}=0$ for $i\neq k$.  
	Thus, defining
	$$\lambda_i=\bigoplus_{j=1}^{s_i}\omega^{b_{ij}}\eta_{n_{ij}}\sigma_i,$$
	we have an isomorphism
	\begin{align*}
		V &\xrightarrow{f} V_1\times \cdots \times V_r \\
		\tx{diag}(X_{11}, \ldots, X_{rr})&\mapsto (X_{11}, \ldots, X_{rr})
	\end{align*}
	
	An entirely similar argument provides an isomorphism 
	$$\varphi:H_\lambda\to H_{\lambda_1}\times\cdots \times H_{\lambda_r},$$
	such that the pair $(f, \varphi)$ is an equivariant isomorphism.  Moreover, 
	$$f\lb\bigoplus_{i=1}^r\bigoplus_{j=1}^{s_i} N_{n_{ij}}\ten I_{s_i}\rb=\lb \bigoplus_{j=1}^{s_1}N_{n_{1j}}\ten I_{s_1}, \ldots, \bigoplus_{j=1}^{s_r}N_{n_{rj}}\ten I_{s_r}\rb.$$
	That is, writing $\mc{O}_{a_1+\cdots a_r}$ for the orbit of $\rec(Q(a_1+\cdots + a_r))$, the geometric parameter of $\rec (Q(a_1+\cdots a_r))$ is $(\mc{O}_{a_1+\cdots + a_r}, \1_{\mc{O}_{a_1+\cdots + a_r}})$, and
	$$f(\mc{O}_{a_1+\cdots + a_r})=\mc{O}_{a_1}\times\cdots \times \mc{O}_{a_r}.$$
	
	Writing $\boxtimes$ for the external tensor product,
	\begin{align*}
		f^\ast(P_{(\mc{O}_{a_1}\times\cdots \times \mc{O}_{a_r}, \boxtimes \1_{a_i})} )&\cong P_{a_1+\cdots + a_r}.
	\end{align*}
	
	Therefore, 
	$f_a(C_b)=C_{b_1}\times\cdots \times C_{b_r}$ and thus
	\begin{align*}
		\dim P_a|_{C_b}&=\dim \lb \boxtimes_{i=1}^rP_{a_i}\rb|_{C_{b_1}\times\cdots \times C_{b_r}}\\
		&=\prod_{i=1}^r\dim P_{a_i}|_{C_{b_i}},
	\end{align*}
	which exactly mirrors the result of Proposition \ref{prop: prodmult}.

	\begin{theorem}[$p$-adic Kazhdan-Lusztig hypothesis for $\GL_n(F)$]\label{thm: pklh}
		Given multisegments $a, b$ such that $Q(a), Q(b)$ have infinitesimal parameter $\lambda:W_F\to\GL_n(\bc)$
		$$[S_Q(a): Q(b)]=\sum_{n\in \bz}\dim \mc{H}^n(P_b)|_{C_a}.$$
	\end{theorem}
	\begin{proof}
		By Lemma \ref{lem: multirays} and Proposition \ref{prop: prodmult}, we know that there exist multisegments $a_1, \ldots a_r, b_1, \ldots, b_r$ such that two segments of $a$ (resp. $b$) belong to the same $a_i$ if and only if they have the same cuspidal support. 
		$$[S_Q(a):Q(b)]=m(a;b)=m(b_1;a_1)\cdots m(b_r; a_r).$$
		
		As each $a_i, b_i$ belongs to a single ray, by Proposition \ref{prop: 1raygeo},
		$$m(b_i;a_i)=\sum_{n\in \bz}\dim \mc{H}^n(P_{b_i})|_{C_{a_i}}.$$
		Thus, by the discussion above,
		\begin{align*}
			[S_Q(a):Q(b)]&=m(b_1;a_1)\cdots m(b_r; a_r)\\
			&=\prod_{i=1}^r\dim\mc{H}^n(P_{b_i})|_{C_{a_i}}\\
			&=\dim \mc{H}^n(\boxtimes_{i=1}^rP_{b_i})|_{\bigtimes_{i=1}^r C_{a_i}} \\
			&=\dim \mc{H}^n(P_b)|_{C_a}. 
		\end{align*}
	\end{proof}



\begin{bibdiv}
		\begin{biblist}
			
			\bib{AchA}{article}{
				title={Combinatorics of Fourier transforms for type A quiver representations},
				author={Achar, Pramod N. and Kulkarni, Maitreyee C. and Matherne, Jacob P.},
				journal={arXiv preprint arXiv:1807.10217},
				date={2018},
				eprint={1807.10217},
			}
			
			\bib{Ari}{article}{
				title={On the decomposition numbers of the Hecke algebra
					of $G(m, 1, n)$},
				author={Ariki, Susumu},
				journal={Journal of Mathematics of Kyoto University},
				volume={36},
				number={4},
				pages={789--808},
				date={1996},
			}
			
			
			\bib{Aub}{article}{
				title={Dualité dans le groupe de Grothendieck de la catégorie des représentations lisses de longueur
					finie d’un groupe réductif p-adique},
				author={Aubert, A.M.},
				journal={Trans. Amer. Math. Soc.},
				volume={347},
				number={6},
				pages={ 2179–2189},
				date={1995},
			}
			
			
			\bib{BCFZ}{article}{
				author={Balodis, Kristaps},
				author={Cunningham, Clifton},
				author={Fiori, Andrew},
				author={Zhang, Qing},
				title={Representations of $p$-adic groups and orbits with smooth closure in a variety of Langlands parameters},
				date ={2025}, 
				eprint={2504.04163},
			}
			
			\bib{BK}{book}{
				title={The Admissible Dual of {$GL(N)$} via Compact Open Subgroups},
				author={Bushnell, Colin J},
				author={Kutzko, Philip C.},
				series={Annals of Mathematics Studies},
				number={129},
				date={1993},
				publisher={Princeton University Press},
				address={Princeton, NJ},
				isbn={9780691021140},
				doi={10.1515/9781400882496}
			}
			
			
			\bib{BB}{article}{
				author = {Beilinson, Alexandre},
				author = {Bernstein, Joseph},
				title = {Localisation de g-modules},
				journal = {Comptes Rendus de l'Académie des Sciences},
				volume = {292},
				number = {1},
				pages = {15-18}
			}
			
			
			
			\bib{BryKas}{article}{
				title={Kazhdan-Lusztig conjecture and holonomic systems},
				author={Brylinski, Jean-Luc },
				author={Kashiwara, M},
				journal={Invent Math},
				volume={64},
				pages={387–410},
				date={1981},
			}
			
			
			
			
			
			\bib{CG}{book}{
				author={Chriss, Neil},
				author={Ginzburg, Victor},
				title={Representation theory and complex geometry},
				series={Modern Birkh\"{a}user Classics},
				note={Reprint of the 1997 edition},
				publisher={Birkh\"{a}user Boston, Ltd., Boston, MA},
				date={2010},
				pages={x+495},
				isbn={978-0-8176-4937-1},
				review={\MR{2838836}},
				doi={10.1007/978-0-8176-4938-8},
			}
			
			
			\bib{CK}{article}{
				author={Chriss, Neil},
				author={Khuri-Makdisi, Kamal },
				title={On the Iwahori-Hecke algebra of a p-adic group},
				journal={International Mathematics Research Notices},
				volume={1998},
				date={1998},
				pages={85 -- 10},
				doi={https://doi.org/10.1155/S1073792898000087},
			}
			
			
			
			\bib{CFZ:cubic}{article}{
				author={Cunningham, Clifton},
				author={Fiori, Andrew},
				author={Zhang, Qing},
				title = {$A$-packets for {$G_2$} and perverse sheaves on cubics},
				journal = {Adv. Math.},
				volume = {395},
				date = {2022},
			}
			
			
			\bib{CFZ:unipotent}{article}{
				author={Cunningham, Clifton},
				author={Fiori, Andrew},
				author={Zhang, Qing},
				title={Toward the endoscopic classification of unipotent representations of p-adic $G_2$},
				date={2022},
				note={\href{https://arxiv.org/abs/2101.04578}{https://arxiv.org/abs/2101.04578}},
			}
			
			
			\bib{CRI}{article}{
				author={Cunningham, Clifton},
				author={Ray, Mishty},
				title={Proof of Vogan's conjecture on Arthur packets: 
					irreducible parameters of p-adic general linear groups},
				date={2022}
				journal={ArXiv preprint arXiv:2206.01027}
			}
			
			
			
			\bib{Cun}{article}{
				author={Cunningham, Clifton L.R.},
				author={Fiori, Andrew},
				author={Moussaoui, Ahmed},
				author={Mracek, James},
				author={Xu, Bin},
				title={Arthur packets for $p$-adic groups by way of microlocal vanishing
					cycles of perverse sheaves, with examples},
				journal={Mem. Amer. Math. Soc.},
				volume={276},
				date={2022},
				number={1353},
				pages={ix+216},
				issn={0065-9266},
				isbn={978-1-4704-5117-2},
				isbn={978-1-4704-7019-7},
				doi={10.1090/memo/1353},
			}
			
			\bib{GH}{article}{
				title={An Introduction to Automorphic Representations: With a view toward trace formulae},
				author={Getz, Jayce R. and Hahn, Heekyoung},
				series={Graduate Texts in Mathematics},
				volume={300},
				date={2024},
				publisher={Springer},
				address={Cham},
				isbn={978-3-031-41151-9},
			}
			
			
			
			
			\bib{Jan}{article}{
				author={Jantzen, Chriss},
				journal = {Annales scientifiques de l'École Normale Supérieure},
				number = {5},
				pages = {527-547},
				publisher = {Elsevier},
				title = {On the Iwahori-Matsumoto involution and applications},
				volume = {28},
				date = {1995}
			}
			
			
			
			
			\bib{KL}{article}{
				title={Representations of Coxeter groups and Hecke algebras},
				author={Kazhdan, D.},
				author={Lusztig, G},
				journal={Invent Math},
				volume={53},
				pages={165--184},
				date={1979},
			}
			
			
			\bib{Kon}{article}{
				title={A note on the Langlands classification and irreducibility of induced representations of p-adic groups},
				author={Konno, Takuya},
				journal={Kyushu Journal of Mathematics},
				volume={57},
				number={2},
				pages={383--409},
				date={2003},
				publisher={Faculty of Mathematics, Kyushu University}
			}
			
			
			
			\bib{PN}{article}{
				author = {Nori, Madhav and Prasad, Dipendra},
				issn = {0075-4102},
				journal = {Journal für die reine und angewandte Mathematik},
				language = {eng},
				number = {762},
				pages = {261-280},
				publisher = {De Gruyter},
				title = {On a duality theorem of Schneider–Stuhler},
				volume = {2020},
				date = {2020},
			}
			
			
			
			\bib{Pro}{article}{
				title={Parabolic Induction Via Hecke Algebras and the Zelevinsky Duality Conjecture},
				author={Kerrigan Procter},
				journal={Proceedings of the London Mathematical Society},
				date={1998},
				volume={77},
			}
			
			\bib{PyvPre}{article}{
				title={Generic Smooth Representations},
				author={Pyvovarov, Alexandre },
				journal={arXiv},
				date={2019},
			}
			
			
			
			\bib{Pyv}{article}{
				title={Generic Smooth Representations},
				author={Pyvovarov, Alexandre },
				journal={Doc. Math.},
				date={2020},
				volume={25},
				pages={2473–-2485},
			}
			
			
			
			\bib{Ree}{article}{
				title={On certain Iwahori invariants in the unramified principal series},
				author={Reeder, Mark},
				journal={Pacific Journal of Mathematics},
				volume={153},
				number={2},
				pages={313--342},
				date={1992},
				publisher={Mathematical Sciences Publishers},
			}
			
			\bib{Rod}{article}{
				title={Repr{\'e}sentations de $ GL (n, k) $ o{\`u} $ k $ est un corps $ p $-adique},
				author={Rodier, Fran{\c{c}}ois},
				journal={S{\'e}minaire Bourbaki},
				volume={24},
				pages={201--218},
				date={1982},
				publisher={Soci{\'e}t{\'e} Math{\'e}matique de France},
			}
			
			
			\bib{Solp}{article}{
				author={Solleveld, Maarten},
				title={Graded Hecke algebras, constructible sheaves and the p-adic Kazhdan--Lusztig conjecture},
				note={\href{https://arxiv.org/pdf/2106.03196.pdf}{https://arxiv.org/pdf/2106.03196.pdf}},
				date={2022}
			}
			
			
			\bib{Sol2}{article}{
				title={On submodules of standard modules}, 
				author={Solleveld, Maarten },
				date={2024},
				note={preprint, \href{https://arxiv.org/abs/2309.10401}{	arXiv:2309.10401}}
			}
			
			
			
			\bib{Tad}{article}{
				author = {Tadi\'c, Marko},
				title = {Classification of unitary representations in irreducible representations of general linear group (non-Archimedean case)},
				journal = {Annales scientifiques de l'\'Ecole Normale Sup\'erieure},
				pages = {335--382},
				publisher = {Elsevier},
				volume = {Ser. 4, 19},
				number = {3},
				date = {1986},
			}
			
			\bib{Vog}{article}{
				author={Vogan, David A., Jr.},
				title={The local Langlands conjecture},
				volume={145},
				publisher={Amer. Math. Soc., Providence, RI},
				date={1993},
				pages={305--379},
				review={\MR{1216197}},
			}
			
			\bib{VogProof}{article}{
				title={The Kazhdan-Lusztig conjecture for real reductive groups},
				author={Vogan, David A},
				booktitle={Representation Theory of Reductive Groups},
				editor={Wallach, Nolan},
				series={Progress in Mathematics},
				volume={53},
				pages={373--460},
				year={1985},
				publisher={Birkh{\"a}user},
				address={Boston, MA},
				doi={10.1007/978-1-4684-9210-6\_12},
			}
			
			\bib{Zel}{article}{
				author={Zelevinski\u{\i}, A. V.},
				title={The $p$-adic analogue of the Kazhdan-Lusztig conjecture},
				language={Russian},
				journal={Funktsional. Anal. i Prilozhen.},
				volume={15},
				date={1981},
				number={2},
				pages={9--21, 96},
				issn={0374-1990},
				review={\MR{617466}},
			}	
			
			\bib{ZelI2}{article}{
				title={Induced representations of reductive $\mf{p}$-adic groups. II. On irreducible representations of $GL(n)$},
				author={Zelevinsky, Andrei V},
				booktitle={Annales scientifiques de l'{\'E}cole Normale Sup{\'e}rieure},
				volume={13},
				number={2},
				pages={165--210},
				date={1980}
			}	
			
			\bib{Zelgnc}{article}{
				title={Two remarks on graded nilpotent classes},
				author={Zelevinsky, A. V.},
				journal={Uspekhi Mat. Nauk},
				volume={40},
				number={1(241)},
				pages={139-140},
				date={1985},
				publisher={IOP Publishing},
			}

		\end{biblist}
	\end{bibdiv}
\end{document}